\documentclass{article}

\usepackage{amssymb,amsmath,amsthm}
\usepackage{texdraw}
\usepackage[latin1]{inputenc}
\usepackage{graphicx}

\def \dt {{\,\mathrm dt}}

\def \d {{\,\mathrm d}}
\def \dtau {{\,\mathrm d\tau}}
\def \dtheta {{\,\mathrm d\vartheta}}

\def \RR {{\mathbb R}}

\def \ZZ {{\mathbb Z}}

\newcommand{\GReU}[3]{\Gamma}%(#1,#2,#3)}

\def \al {\alpha}

\def \ep {\varepsilon}
\def \eps {\varepsilon}
\def \vt {\vartheta}
\def \vp {\varphi}

\def \action{\mathcal{A}}

\def \ell{\mathcal{E}}
\def \cont{\mathcal{C}}
\def \morse{\mathcal{M}}
\def \pjump {\Delta_{\mathrm{pos}}}
\def \vjump {\Delta_{\mathrm{vel}}}

\newcommand{\Uh}{\mathcal{U}}
\newcommand{\Vh}{\mathcal{V}}

\newcommand{\cotan}{\mathrm{cotan}}

\newtheorem{theorem}{Theorem}[section]
\newtheorem{lemma}[theorem]{Lemma}

\newtheorem{definition}[theorem]{Definition}
\newtheorem{proposition}[theorem]{Proposition}
\newtheorem{remark}[theorem]{Remark}
\newtheorem{corollary}[theorem]{Corollary}

\title{\sc Entire minimal\\ parabolic trajectories:\\ the planar anisotropic\\ Kepler problem}
\date{}
\author{%
Vivina Barutello\footnote{Dipartimento di Matematica, Universit\`a degli Studi
di Torino, Via Carlo Alberto, 10,  10123 Torino,
Italy. e-mail: \texttt{vivina.barutello@unito.it}}
\and
Susanna Terracini\footnote{Dipartimento di Matematica e Applicazioni, Universit\`a degli Studi
di Milano-Bicocca, Via Bicocca degli Arcimboldi, 8, 20126 Milano,
Italy. e-mail: \texttt{susanna.terracini@unimib.it}}
\and
Gianmaria Verzini\footnote{Dipartimento di Matematica, Politecnico di Milano, Piazza
Leonardo da Vinci, 32,  20133 Milano, Italy.
e-mail: \texttt{gianmaria.verzini@polimi.it}}
}
\begin{document}
%%%%%%%%%%%%%%%%%%%%%%%%%%%%%%%%%%%%%%%%%%%%%%%%%%%%%%%%%%%%%%%%%%%%%%%%%%%%%%%
\maketitle
%%%%%%%%%%%%%%%%%%%%%%%%%%%%%%%%%%%%%%%%%%%%%%%%%%%%%%%%%%%%%%%%%%%%%%%%%%%%%%%

\begin{abstract}
We continue the variational approach to parabolic
trajectories introduced in our previous paper \cite{BTV}, which sees parabolic orbits as
minimal phase transitions.

We  deepen and complete the analysis in the planar case for homogeneous singular potentials. We
characterize all parabolic orbits connecting
two minimal central configurations as free-time Morse minimizers (in a given homotopy
class of paths). These may occur for at most one value of the homogeneity exponent.
In addition, we  link this threshold of existence of parabolic
trajectories with the absence of collisions for all the minimizers of fixed-ends
problems. Also the existence of action minimizing periodic trajectories with nontrivial
homotopy type  can be related with the same threshold.
\end{abstract}

%======================================
\section{Introduction and Main Results}\label{sec:intro}
%======================================

For a positive, singular potential $V \in \cont^2(\RR^d \setminus \{0\})$,
vanishing at infinity, we study the Newtonian system
\begin{equation}\label{eq:dynsys}
\ddot x (t) =\nabla V (x(t)),
\end{equation}
searching for \emph{parabolic} solutions, i.e. entire solutions satisfying the zero-energy relation
\begin{equation}\label{eq:dynsys2}
\frac12 |\dot x(t)|^2 = V(x(t)), \qquad \text{for every } t \in \RR.
\end{equation}
In the Kepler problem ($V(x)=1/|x|$) all global zero-energy trajectories are indeed parabola.
In this paper we are concerned with $(-\alpha)$-homogeneous potentials, with $\alpha \in (0,2)$.
Within this class of potentials, parabolic trajectories are homoclinic to infinity, which represents the minimum
of the potential.

In celestial mechanics, and more in general in the theory of singular hamiltonian
systems, parabolic trajectories play a central role and they are known to carry precious
information on the behavior of general solutions near collisions. On the other hand,
parabolic trajectories are structurally unstable and therefore
are usually considered beyond the range of application of variational or other
global methods.
In spite of this, in our previous paper \cite{BTV}, we introduced a new variational
approach to their existence as
minimal phase transitions.

The purpose of the present paper is to deepen and complete the analysis in the
planar case $d=2$: we will succeed in characterizing all parabolic orbits connecting
two minimal central configurations as free-time Morse minimizers (in a given homotopy
class of paths). In addition, we shall link the threshold of existence of parabolic
trajectories with the absence of collisions for all the minimizers of fixed-ends
problems. Also the existence of action minimizing periodic trajectories with nontrivial
homotopy type  will be related with the same threshold.

In the plane $\RR^2$ we use the polar coordinates $x=(r\cos \vt,r\sin\vt)=(r,\vt)$ (despite the ambiguous notation, it will always be clear from the
context wether a pair denotes either cartesian or polar coordinates).
Under this notation any $(-\alpha)$-homogeneous potential $V$ can be written as
\[
V(x) = \frac{U(\vt)}{r^\alpha},
\]
where
\[
U(\vt) := V(\cos \vt, \sin \vt).
\]
The potential $V$ is then a generalization of the \emph{anisotropic Kepler potential}
(extensively studied for instance in \cite{DevInvMath1978,DevProgMath1981,Gutzwiller73,Gutzwiller77,Gutzwiller81}),
which actually corresponds to the value $1$ of the parameter $\alpha$ and a specific $U$.
For such potentials, it is well known that parabolic trajectories admit in/outgoing asymptotic directions which are necessarily
critical points of $U(\vt)$: these are called \emph{central configurations}. We are mostly interested to parabolic trajectories
connecting two \emph{minimal} central configurations. To be more precise, given
\[
0 \leq \vt_1 \leq \vt_2 < 2\pi,
\]
we define the sets of potentials
\[
\Uh = \Uh_{\vt_1\vt_2} := \left\{ U\in \cont^2(\RR) : \begin{array}{l}
                                    \text{for every } \vt \in \RR \text{ and } i=1,2\\
                                    U(\vt+2\pi)=U(\vt)\\
                                    U(\vt) \geq U(\vt_{1})=U(\vt_{2})>0\\
                                    U''(\vt_{i})>0
                                   \end{array} \right\},
\]
and, with a slight abuse of notation,
\[
\begin{split}
\Vh :=& \left\{ V = (U,\alpha) : U \in \Uh \text{ and } \alpha \in (0,2) \right\} \\
     =& \left\{ V \in \cont^2(\RR^{2}\setminus \{0\}) :
     V(x) = \frac{U(\vt)}{r^{\alpha}},
     U \in \Uh \text{ and } \alpha \in (0,2) \right\}.
\end{split}
\]
For a given $V \in \Vh$, we introduce the action functional
\[
\action(x) = \action([a,b];x):=
\int_{a}^{b} \frac12 |\dot x(t)|^2 + V(x(t)) \, \dt.
\]

In our previous paper \cite{BTV}, we introduced the set of Morse parabolic minimizers
associated to $\action$ and having asymptotic directions
$\xi^- = (\cos \vt_{1},\sin \vt_{1})$ and $\xi^+ = (\cos \vt_{2},\sin \vt_{2})$. Nonetheless, since $\RR^2\setminus\{0\}$ is not simply connected,
as a peculiar fact in the planar case one can also impose a topological constraint
in the form of a homotopy class for the minimizer, for example imposing $h\in\ZZ$ counterclockwise rotations around the origin.
Lifting such a trajectory to the universal covering of $\RR^2\setminus\{0\}$, this corresponds to joining $\vt_1$ with $\vt_2+2h\pi$. Motivated by these considerations, we introduce the set
\[
\Theta = \Theta_{\vt_1\vt_2} := \left\{ \vt \in \RR : \vt = \vt_i +2n\pi
                            \text{ for some } n \in \ZZ \text{ and } i\in\{1,2\} \right\}
\]
and, given $\vt^- \neq \vt^+$ in $\Theta_{\vt_1\vt_2}$ (or, more in general,
$\vt^- \neq \vt^+$ central configurations), we define the following class of paths.
\begin{definition}\label{defi:Morse_min}
We say that $x =(r,\vt) \in H^1_{\mathrm{loc}}(\RR)$ is a
\emph{parabolic trajectory} associated with $\vt^-$, $\vt^+$ and $V$, if it
satisfies equations \eqref{eq:dynsys}, \eqref{eq:dynsys2} and
\begin{itemize}
 \item $\min_{t\in\RR} r(t) >0$;
 \item $r(t)\to+\infty$, $\vt(t)\to\vt^\pm$ as $t\to\pm\infty$;
\end{itemize}
We say that $x$ is a (free time) \emph{parabolic Morse minimizer} if moreover there holds
\begin{itemize}
 \item for every $t_1<t_2$, $t_1'<t_2'$, and $z=(\rho,\zeta)\in H^1(t_1',t_2')$,
 there holds
\[
\rho(t_i')=r(t_i),\ \zeta(t_i')=\vt(t_i),\ i=1,2 \quad\implies\quad \action([t_1,t_2];x)\leq\action([t_1',t_2'];z).
%z(t_i')=x(t_i)\qquad\implies\qquad \action([t_1,t_2];x)\leq\action([t_1',t_2'];z)
\]
(this last property actually implies \eqref{eq:dynsys}, \eqref{eq:dynsys2}). A fixed time minimizer fulfills the above minimality condition
only with $t_i'=t_i$.
\end{itemize}
\end{definition}
 Under the previous definition the following holds.
\begin{theorem}\label{theo:main}
Let $U \in \Uh$ and $\vt^-, \vt^+ \in \Theta$, $\vt^- \neq \vt^+$ be fixed minimal central configurations; then
\begin{itemize}
\item there exists at most one $\bar \alpha = \bar \alpha
(\vt^-,\vt^+,U)\in(0,2)$ such that $V=(U,\alpha)$ admits  a corresponding
parabolic trajectory associated with $(\vt^-,\vt^+,U)$ if and only if $\alpha = \bar \alpha$;
\item every parabolic trajectory associated with $\vt^-$, $\vt^+$ and $U$ is a free time Morse minimizer;
\item if $|\vt^+ - \vt^-| > \pi$ then there exists exactly one $\bar \alpha$ such that $V=(U,\alpha)$ admits
a corresponding parabolic Morse minimizer if and only if $\alpha = \bar \alpha$.
\end{itemize}
\end{theorem}
Let us point out that, if $|\vt^+-\vt^-| \leq \pi$, such a number
$\bar\alpha(\vt^-,\vt^+,U)$ may or may not exist depending on the properties of $U$.

To proceed with the description of our results, let us extend the function $\bar\alpha(\vt^-,\vt^+,U)$ to the whole of the possible triplets $(\vt^-,\vt^+,U)\in\Theta\times\Uh$ by setting its value to
zero if there are no parabolic trajectories for any $\alpha$.
This exponent can be related to the presence/absence of collisions for both the fixed time and the free time Bolza problems within the sector defined by the angles $\vt^-$ and $\vt^+$.

The problem of the exclusion of collisions for action minimizing trajectories has nowadays a long history, starting from the first elaborations in the late eighties, e.g. \cite{AC-Z,DegGiaMar1987,DegGiaMar1988,CZelSer1992,CZelSer1994,SerraTer94,Tanak2000} up to the extensive researches of the last decade, mostly motivated by the search of new symmetric collisionless periodic solutions to the $n$--body problem (e.g. \cite{ChMont1999,ChenVen2000,Chen2001, Fer2007}). Starting from the idea of averaged variation by Marchal \cite{Marchal01, chencinerICM02}, later made fully rigorous, extended and refined in \cite{FT2003}, a rather complete analysis of the possible singularities of minimizing trajectories has been recently achieved in \cite{BFT2}. In the literature, minimal parabolic trajectories have been studied in connection
with the absence of collisions for fixed-endpoints minimizers.
More precisely, as remarked by Luz and Maderna in \cite{LuzMad2011},
the property to be collisionless for all Bolza minimizers implies the absence of
parabolic trajectories which are Morse minimal for the usual $n$--body problem with $\alpha=1$. On the contrary,  minimal parabolic arcs (i.e.,
defined only on the half line) exist for every starting configuration,
as proved by Maderna and Venturelli in \cite{MadVen2009}.

A special attention has been devoted to  minimizers subject to topological constraints and
to the existence of trajectories having a particular homotopy type (see e.g.
\cite{Gordon77,Mont1998,AriGazTer2000,Marchal01,TerVen07,Chen2008}). For such constrained
minimizers the averaged variation technique is not available, and other devices have to be
designed to avoid the occurrence of collisions. Starting from \cite{TerVen07}, motivated by
the search of periodic solutions having prism symmetry, a connection has been established
between the apsidal angles of parabolic trajectories and the exclusion of collisions for
minimizers with a given rotation angle. In fact we can  now draw a complete picture of the
role played by the parabolic orbits in the solution of the collision-free minimization
problem with fixed ends.
\begin{definition}\label{defi:Bolza_min}
Given a potential $V$, we say that $x =(r,\vt) \in H^1(t_1,t_2)$ is a
\emph{fixed-time Bolza minimizer} associated to the ends $x_1=r_1e^{i\varphi_1}$, $x_2=r_2e^{i\varphi_2}$, 
if
\begin{itemize}
 \item $r(t_i)=r_i$ and $\vt(t_i)=\varphi_i$, $i=1,2$;
 \item for every $z=(\rho,\zeta)\in H^1(t_1,t_2)$,  there holds
\[
\rho (t_i)=r_i,\ \zeta(t_i)=\varphi_i,\ i=1,2\quad\implies\quad \action([t_1,t_2];x)\leq\action([t_1,t_2];z).
\]
\end{itemize}
If $\min_{t\in[t_1,t_2]} r(t) >0$ we say that the Bolza minimizer is collisionless.
\end{definition}

\begin{theorem}\label{theo:bolza}
Let $U\in \Uh$, $\vt^- \neq \vt^+ \in \Theta$, and consider a perturbed
potential $V=\dfrac{U(\vt)}{r^\alpha}+W$, with  $V\in\mathcal C^1(\RR^2\setminus{0})$, $\alpha>\alpha'$ and
\begin{equation}\label{eq:perturbed}
\lim_{r\to 0} r^{\alpha'} \left(W(x)+r|\nabla W(x)|\right)=0\;.
 \end{equation}
If $\alpha>\bar\alpha(U, \vt^-,\vt^+)$ then all
fixed-time Bolza minimizers associated to $x_1=(r_1,\vp_1)$ and $x_2=(r_2,\vp_2)$
within the sector $[\vt^-,\vt^+]$ are collisionless.
\end{theorem}
It is worthwhile noticing that, if conversely $\alpha\leq\bar\alpha(U, \vt^-,\vt^+)$, then there are always some
Bolza problems which admit only colliding minimizers. In addition, the very same arguments imply, when
$\alpha=\bar\alpha(U, \vt^-,\vt^+)$, the following statement, which gives a variational
generalization of Lambert's Theorem on the existence of the direct and inverse arcs for the planar Kepler problem
(\cite{Marchal01,Whittaker}).
\begin{proposition}\label{propo:1.5}
Let $U\in \Uh$, $\vt^- \neq \vt^+ \in \Theta$, and $V$ be a perturbed potential as in the previous theorem,
with $\alpha=\bar\alpha(U, \vt^-,\vt^+)$.
Given any pair of points $x_1$ and $x_2$ in the sector $(\vt^-,\vt^+)$,  all
fixed-time Bolza minimizers associated to $x_1$, $x_2$  within the sector $[\vt^-+\ep,\vt^+-\ep]$,
for some $\ep>0$, are collisionless.
\end{proposition}
Some further interesting consequences can be drawn, in the special case when
$\vt^+=\vt^-+2k\pi$, which connect the parabolic threshold with the existence of
non-collision periodic orbits having a prescribed winding number (this is connected with
the minimizing property of Kepler ellipses, see \cite{Gordon77}).

\begin{theorem}\label{theo:gordon}
Let $U\in \Uh$ be such that all its local minima are non-degenerate global ones,
and consider the  potential $V=\dfrac{U(\vt)}{r^\alpha}$.
Given any integer $k\neq 0$ and period $T>0$, if
\begin{equation}\label{eq:alfamax}
\alpha>\bar\alpha(U, \vt^*,\vt^*+2k\pi)\;,\quad\text{for every minimum $\vt^*$ of $U$,}
\end{equation}
then any action minimizer in the class of $T$--periodic trajectories winding $k$ times around zero is collisionless.
\end{theorem}
The outline of the paper is the following: in Section \ref{sec:devaney} we
exploit some results due to Devaney \cite{DevInvMath1978,DevProgMath1981} in
order to rewrite equations \eqref{eq:dynsys}, \eqref{eq:dynsys2} in terms of an equivalent planar first-order system; this allows us to develop a first
phase-plane analysis of the dynamical properties of parabolic trajectories. In
Section \ref{sec:min_prop} we turn to the variational properties of zero-energy
solutions. In Section \ref{sec:sector} we prove Theorem \ref{theo:main} in the particular case in which $\pi< \vt^+-\vt^-\leq 2\pi$. Finally
Sections \ref{sec:srotolo} and \ref{sec:other_proofs} are devoted to the
end of the proof of Theorem \ref{theo:main} and to the proofs of Theorems \ref{theo:bolza}, \ref{theo:gordon}, respectively.
%
%==============================
\section{Phase Plane Analysis}\label{sec:devaney}
%==============================
%
Following Devaney \cite{DevInvMath1978,DevProgMath1981}, an appropriate change of variables makes the differential
problem \eqref{eq:dynsys}, \eqref{eq:dynsys2} equivalent to a planar first order  system, for which a phase plane
analysis can be carried out. This allows a first investigation of its trajectories from a dynamical
(i.e. not variational) point of view.

Let $U\in\Uh_{\vt_1\vt_2}$, and let us assume for simplicity that $U$ is a Morse function, even
though the only important assumption is that $\vt_1$, $\vt_2$ are non-degenerate.
Introducing the Cartesian coordinates $q_1=r\cos \vt$, $q_2=r\sin\vt$
and the momentum vector $(p_1,p_2) = (\dot q_1,\dot q_2)$, we write
equations \eqref{eq:dynsys} and \eqref{eq:dynsys2} as
\[
\begin{cases}
\dot q_1 = p_1 \\
\dot q_2 = p_2 \\
\dot p_1 = {\partial_{q_1}}\left( {r^{-\alpha}}{U(\vt)} \right)
         = {r^{-\alpha-2}}\left( -U'(\vt)q_2 -\alpha U(\vt)q_1 \right) \\
\dot p_2 = {\partial_{q_2}}\left( {r^{-\alpha}}{U(\vt)} \right)
         = {r^{-\alpha-2}}\left( U'(\vt)q_1 -\alpha U(\vt)q_2 \right),
\end{cases}
\]
and
\[
\frac12 \left( p_1^2 + p_2^2\right) = \frac{U(\vt)}{r^\alpha}.
\]
Since $U(\vartheta)\geq U(\vt_1)=U(\vt_2) =: U_{\min}>0$, we have that $|p|\neq0$.
As a consequence, for every solution of the previous dynamical system we can find smooth functions
$z >0$ and $\vp \in \RR$ in such a way that
$p_1=r^{-\alpha/2}z\cos \vp$, $p_2=r^{-\alpha/2}z\sin\vp$.
These functions satisfy
\[
z = \sqrt{2U(\vt)}
\]
and
\[
\begin{cases}
\dot r   = r^{-\alpha/2}z\left( \cos \vt \cos \vp + \sin \vt \sin \vp \right)
         = r^{-\alpha/2}z \cos (\vp-\vt) \\
\dot \vt = r^{-1-\alpha/2}z\left( \cos \vt \sin \vp - \sin \vt \cos \vp \right)
         = r^{-1-\alpha/2}z\sin (\vp-\vt) \\
\dot z   = r^{-1-\alpha/2}U'(\vt)\sin(\vp-\vt) \\
\dot \vp = \frac{1}{z} r^{-1-\alpha/2} \left[ U'(\vt)\cos(\vp-\vt) +\alpha U(\vt) \sin(\vp-\vt)\right].
\end{cases}
\]
This system has a singularity at $r=0$ that can be removed by a change of time scale.
Assuming $r>0$, we introduce the new variable $\tau$ via
\[
\frac{\dt}{\dtau} = zr^{1+\alpha/2}
\]
in order to rewrite the dynamical system as (here ``\,$'$\,'' denotes the derivative with respect to $\tau$)
\begin{equation}\label{eq:syscomplete}
\begin{cases}
r'   = r z^2 \cos (\vp-\vt) = 2r U(\vt) \cos (\vp-\vt) \\
z'   = z U'(\vt)\sin(\vp-\vt) \\
\vt' = z^2 \sin (\vp-\vt) = 2 U(\vt) \sin (\vp-\vt) \\
\vp' = U'(\vt)\cos(\vp-\vt) +\alpha U(\vt) \sin(\vp-\vt),
\end{cases}
\end{equation}
which contains the independent planar system
\begin{equation}\label{eq:dydthetaphi}
\begin{cases}
\vt' = 2 U(\vt) \sin (\vp-\vt) \\
\vp' = U'(\vt)\cos(\vp-\vt) +\alpha U(\vt) \sin(\vp-\vt).
\end{cases}
\end{equation}
It is immediate to see that the systems above enjoy global existence, and that the stationary points of \eqref{eq:dydthetaphi} are the points $(\vt^*,\vp^*)$, where $U'(\vt^*)=0$ and $\sin(\vp^*-\vt^*)=0$.

\begin{theorem}[Devaney \cite{DevProgMath1981}]\label{thm:devaney}
The path $x=x(t)$ satisfies \eqref{eq:dynsys}, \eqref{eq:dynsys2} if and only if
$(\vt,\vp)$ satisfies \eqref{eq:dydthetaphi} (and $(r,z)$ satisfies \eqref{eq:syscomplete}).

The function
\[
v(\tau) = \sqrt{U(\vartheta(\tau))} \cos\left( \varphi(\tau)-\vartheta(\tau) \right),
\]
is non-decreasing on the solutions of \eqref{eq:dydthetaphi}, which correspond to
\begin{itemize}
\item saddle-type equilibria $(\vt^{*},\vt^*+h\pi)$, $U'(\vt^*)=0$,  $U''(\vt^*)>0$  and $h\in\ZZ$;
\item sink/source-type equilibria $(\vt^{*},\vt^*+h\pi)$, where $U'(\vt^*)=0$,  $U''(\vt^*)<0$  and $h\in\ZZ$;
\item heteroclinic trajectories connecting two of the previous equilibria.
\end{itemize}
To every trajectory of \eqref{eq:dydthetaphi} there corresponds infinitely
many trajectories of \eqref{eq:syscomplete}, all equivalent through a radial homotheticity.

The corresponding solutions of \eqref{eq:dynsys}, \eqref{eq:dynsys2} satisfy the following:
\begin{itemize}
\item if
\begin{equation}\label{eq:cond_cos}
\text{$\cos(\vp-\vt)\to\pm1$ as $\tau\to\pm\infty$,}
\end{equation}
then $x$ is globally defined and unbounded in the future/past (in $t$);
\item if $\cos(\vp-\vt)\to\mp 1$ as $\tau\to\pm\infty$, then $t(\tau)\to T_{\pm}\in\RR$
and $x(t)\to0$ as $t\to T_{\pm}$.
\end{itemize}
\end{theorem}

\begin{figure}[!t]
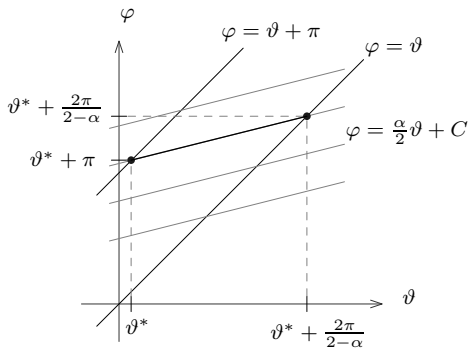

\begin{center}
\begin{texdraw}
\drawdim cm  \setunitscale .5
%% frecce-assi %%
\linewd 0.01 \setgray 0.2
\move (0 -1) \arrowheadtype t:V \arrowheadsize l:0.3 w:0.2 \avec (0 7)
\htext (0 7.5) {\footnotesize{$\varphi$}}
\move (-1 0) \arrowheadtype t:V \arrowheadsize l:0.3 w:0.2 \avec (7 0)
\htext (7.5 0) {\footnotesize{$\vt$}}
%% rette punti critici %%
\linewd 0.02 \setgray 0
\move (-.6 -.6) \lvec (6.5 6.5)
\htext (6.5 6.6) {\footnotesize{$\vp=\vt$}}
\move (-.6 2.9) \lvec (3.3 6.8)
\htext (2.7 7) {\footnotesize{$\vp=\vt+\pi$}}
%% capisaldi asse orizzontale/verticale
\linewd 0.01
\move (.33 -0.2) \lvec (.33 0.2)
\move (5 -0.2) \lvec (5 0.2)
\htext (.13 -.8) {\footnotesize{$\vt^*$}}
\htext (4 -1.2) {\footnotesize{$\vt^*+\frac{2\pi}{2-\alpha}$}}
\move (-0.2 3.83) \lvec (0.2 3.83)
\move (-0.2 5) \lvec (0.2 5)
\htext (-2.4 3.6) {\footnotesize{$\vt^*+\pi$}}
\htext (-3 4.8) {\footnotesize{$\vt^*+\frac{2\pi}{2-\alpha}$}}
%% tratteggi verticali e orizzontali
\linewd 0.01 \setgray 0.5 \lpatt (0.2 0.2)
\move (.33 0.2) \lvec (.33 3.83) \lvec (0.2 3.83)
\move (5 0.2) \lvec (5 5) \lvec (0.2 5)
%% fascio traiettorie
\lpatt()
\move (-.25 3.68) \lvec (6 5.25)
\move (-.25 4.68) \lvec (6 6.25)
\move (-.25 2.68) \lvec (6 4.25)
\htext (6 4.25) {\footnotesize{$\vp = \frac{\alpha}{2}\vt+C$}}
\move (-.25 1.68) \lvec (6 3.25)
\setgray 0
%% punti di equilibro
\linewd 0.03
\move (.33 3.83) \fcir f:0.1 r:0.1
\move (5 5) \fcir f:0.1 r:0.1
%% eteroclina
\linewd 0.04
\move (.33 3.83) \lvec (5 5)
\end{texdraw}
\end{center}
\caption{the figure sketches the phase portrait of \eqref{eq:dydthetaphi}
when $U(\vartheta)\equiv 1$. The dynamical system reads
$\vp' = ({\alpha}/{2})\,\vt' = \alpha \sin(\vp-\vt)$,
which critical points satisfy $\varphi = \vartheta +k\pi$, $k \in \ZZ$.
Trajectories lie on the bundle $\varphi = (\alpha/2) \vartheta + C$, $C \in \RR$, and,
recalling condition \eqref{eq:cond_cos}, we deduce that parabolic solutions coincide
with heteroclinic connections departing from points on
$\varphi = \vartheta +(2k+1)\pi$ and ending on $\varphi = \vartheta +2k\pi$, for some $k\in\ZZ$.
For instance, when $k=0$, we obtain heteroclinics connecting $(\vartheta^*,\vartheta^*+\pi)$ to
$(2\pi/(2-\alpha)+\vartheta^*,2\pi/(2-\alpha)+\vartheta^*)$, for some $\vartheta^* \in \RR$.
Going back to the original dynamical system, this implies that parabolic motions exists
only when the angle between the ingoing and outgoing asymptotic directions
is $2\pi/(2-\alpha)$; let us emphasize that such angle is always greater than $\pi$.
When $\alpha =1$, i.e. in the classical Kepler problem, this angle is $2\pi$:
the heteroclinic between $(\vartheta^*,\vartheta^*+\pi)$ and
$(2\pi+\vartheta^*,2\pi+\vartheta^*)$
actually describes a parabola whose axis form an angle
$\vartheta^*$ with the horizontal line.
\label{fig:keplero}}
\end{figure}

In Figure \ref{fig:keplero} we describe the phase plane for the dynamical system \eqref{eq:dydthetaphi}
when $U$ is \emph{isotropic} and in particular for the Kepler problem.
On the other hand, if we take into account an anisotropic potential $U$ in the class $\Uh_{\vt_1 \vt_2}$
and a homogeneous extension $(U,\alpha)$, $\alpha \in (0,2)$,
then we can deduce the following result (by time reversibility, it is not
restrictive to assume that $\vt^-<\vt^+$).
\begin{corollary}\label{coro:as}
Let $\vt^-<\vt^+$ belong to $\Theta_{\vt_1 \vt_2}$ and let $x=rs$ be an associated  parabolic Morse minimizer
for $(U,\alpha)$. Then (a suitable choice of)
the corresponding $(\vt,\vp)$ is an heteroclinic connection
between the saddles
\[
(\vt^-,\vt^- + \pi) \text{ and } (\vt^+,\vt^+).
\]
Moreover $\vt$ is strictly increasing between $\vt^-$ and $\vt^+$.
\end{corollary}
\begin{proof}
Since $\vt^{\pm}$ are minima for $U$ we have that $(\vt,\vp)$ connects
the two saddles (say)
\[
(\vt^-,\vt^- + h_1\pi) \text{ and } (\vt^+,\vt^+ + h_2\pi),
\]
in such a way that
\[
\lim_{\tau\to-\infty}\left[\vp(\tau)-\vt(\tau)\right]=h_1\pi,\qquad
\lim_{\tau\to+\infty}\left[\vp(\tau)-\vt(\tau)\right]=h_2\pi.
\]
Since $x$ is globally defined, condition \eqref{eq:cond_cos} holds,
yielding $\cos (h_1\pi) = -1$ and $\cos (h_2\pi) = 1$, that is $h_1$ is odd while $h_2$ is even.
Since $v$ is non-decreasing, we have that
\[
-\sqrt{U_{\min}} = v(-\infty) < v(\tau) < v(+\infty) = \sqrt{U_{\min}}.
\]
Now we observe that
\begin{equation}\label{eq:v'}
v' = (2-\alpha)\left[U(\vartheta)\right]^{3/2} \sin^2\left( \varphi-\vartheta \right)
   = (2-\alpha)\sqrt{U(\vartheta)} \left[U(\vartheta)-v^2\right],
\end{equation}
hence $v$ strictly increases.
Then $\sin \left( \varphi-\vartheta \right) \neq 0$, therefore also
$\vt$ is strictly monotone. Since $\vt^- < \vt^+$ we obtain that $\vt$ increases.
But this finally implies that $\sin \left( \varphi-\vartheta \right) > 0$, for every $\tau$.
Summing up all the information we deduce that
\[
h_1 = h_2 + 1. \qedhere
\]
\end{proof}
Motivated by the previous result we devote the rest of the section to study the properties
of the stable and unstable trajectories associated to the saddle points of \eqref{eq:dydthetaphi},
in dependence of the parameter $\alpha$. To start with,
using equation \eqref{eq:v'}, we provide a necessary condition for the existence of saddle-saddle connections.
\begin{lemma}\label{lem:stimaalpha}
Let us assume that for some $\alpha \in (0,2)$ there exists a saddle-saddle connection
for \eqref{eq:dydthetaphi} between $(\vt^-,\vt^-+\pi)$ and $(\vt^+,\vt^+)$.
Then
\[
2-\frac{2\pi}{\vt^+-\vt^-} \leq \alpha \leq 2-\frac{4}{\vt^+-\vt^-}\arcsin \sqrt\frac{U_{\min}}{U_{\max}},
\]
where $U_{\min}\leq U(\vt)\leq U_{\max}$, for every $\vt$.
\end{lemma}
\begin{proof}
Let $(\vt,\vp)$ be such an heteroclinic.
Reasoning as in the proof of the previous corollary,
one can deduce that both $v$ and $\vt$ are (strictly) monotone in $\tau$.
It is then possible to write $v = v(\tau(\vt)) =: \hat v(\vt)$ obtaining that
\[
\lim_{\vt \to \vt^{\pm}} \hat v(\vt) = \pm\sqrt{U_{\min}}.
\]
With this notation we can write
\[
\frac{\d \hat v}{\dtheta} = v'(\tau) \frac{\dtau}{\dtheta}
                     = \frac{2-\alpha}{2}\frac{\sqrt{U(\vartheta)}}{U(\vt)}
                       \frac{U(\vartheta)-v^2}{\sin (\vp-\vt)}
                     = \frac{2-\alpha}{2}\sqrt{U(\vartheta)-\hat v^2}.
\]
Integrating on $\vartheta \in [\vt^-,\vt^+]$, we obtain
on one hand
\begin{equation} \label{eq:stimadx}
\vt^+ - \vt^-
\leq \frac{2}{2-\alpha}\int_{-\sqrt{U_{\min}}}^{\sqrt{U_{\min}}} \frac{\d v}{\sqrt{U_{\min}-v^2}}
=    \frac{2\pi}{2-\alpha}
\end{equation}
and on the other hand
\begin{equation} \label{eq:stimasx}
\vt^+ - \vt^-
\geq \frac{2}{2-\alpha}\int_{-\sqrt{U_{\min}}}^{\sqrt{U_{\min}}} \frac{\d v}{\sqrt{U_{\max}-v^2}}
=    \frac{4}{2-\alpha}\arcsin \sqrt\frac{U_{\min}}{U_{\max}}.\qedhere
\end{equation}
\end{proof}
Using the previous arguments, together with
standard results in structural stability, it is already
possible, for appropriate values of $\alpha$, to show the existence of saddle-saddle
heteroclinic connections (see Figure \ref{fig:pplane}).
In any case, if in principle saddle-saddle connections occur only for particular values of $\alpha$,
on the other hand, whenever $\vt^\pm$ are minima for $U$, \emph{for every} $\alpha$
they correspond to saddle points. The above techniques allow us to study the dependence
of their stable and unstable manifolds on $\alpha$.
\begin{figure}[!t]
\centering
\includegraphics[width=4.5cm]{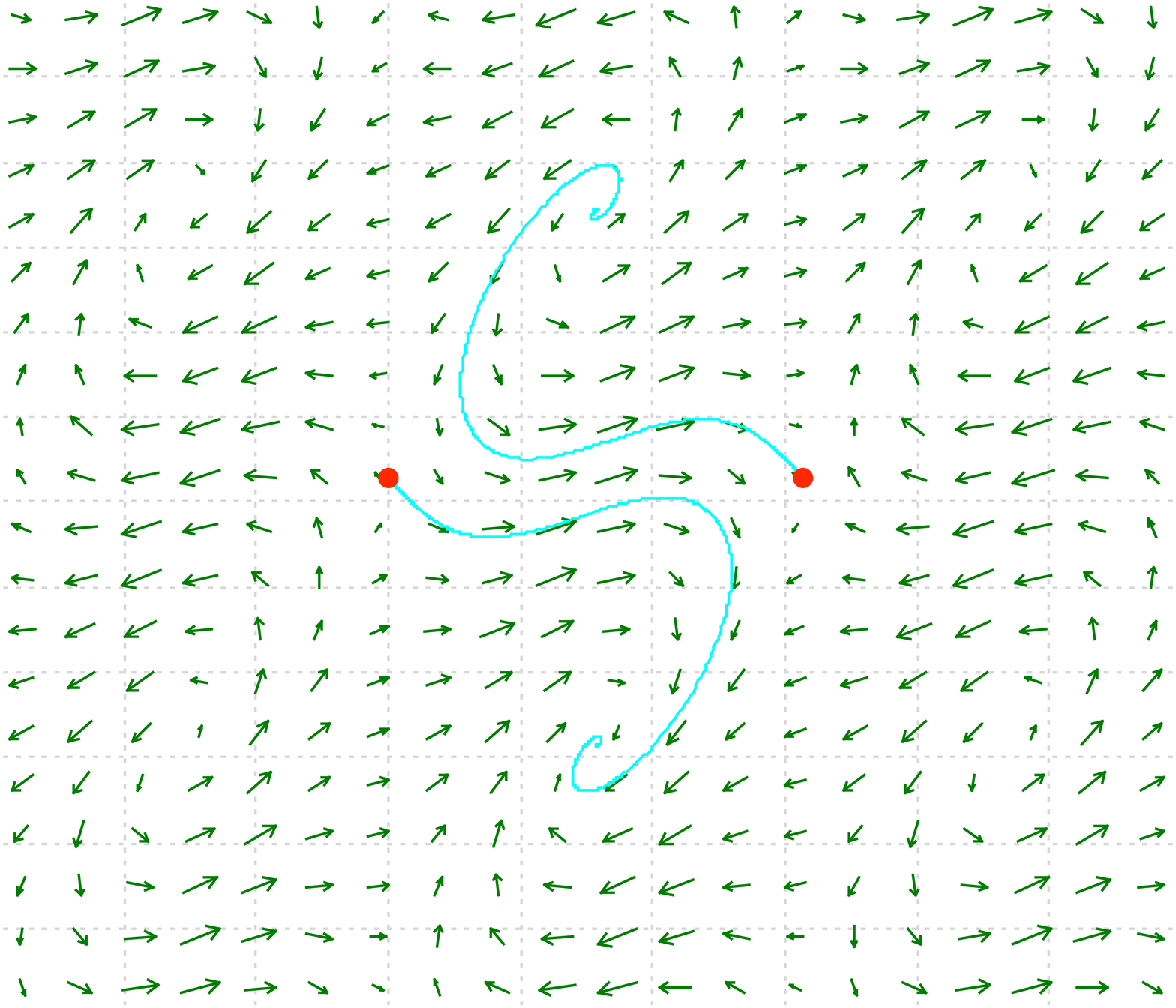}
\qquad
\includegraphics[width=4.5cm]{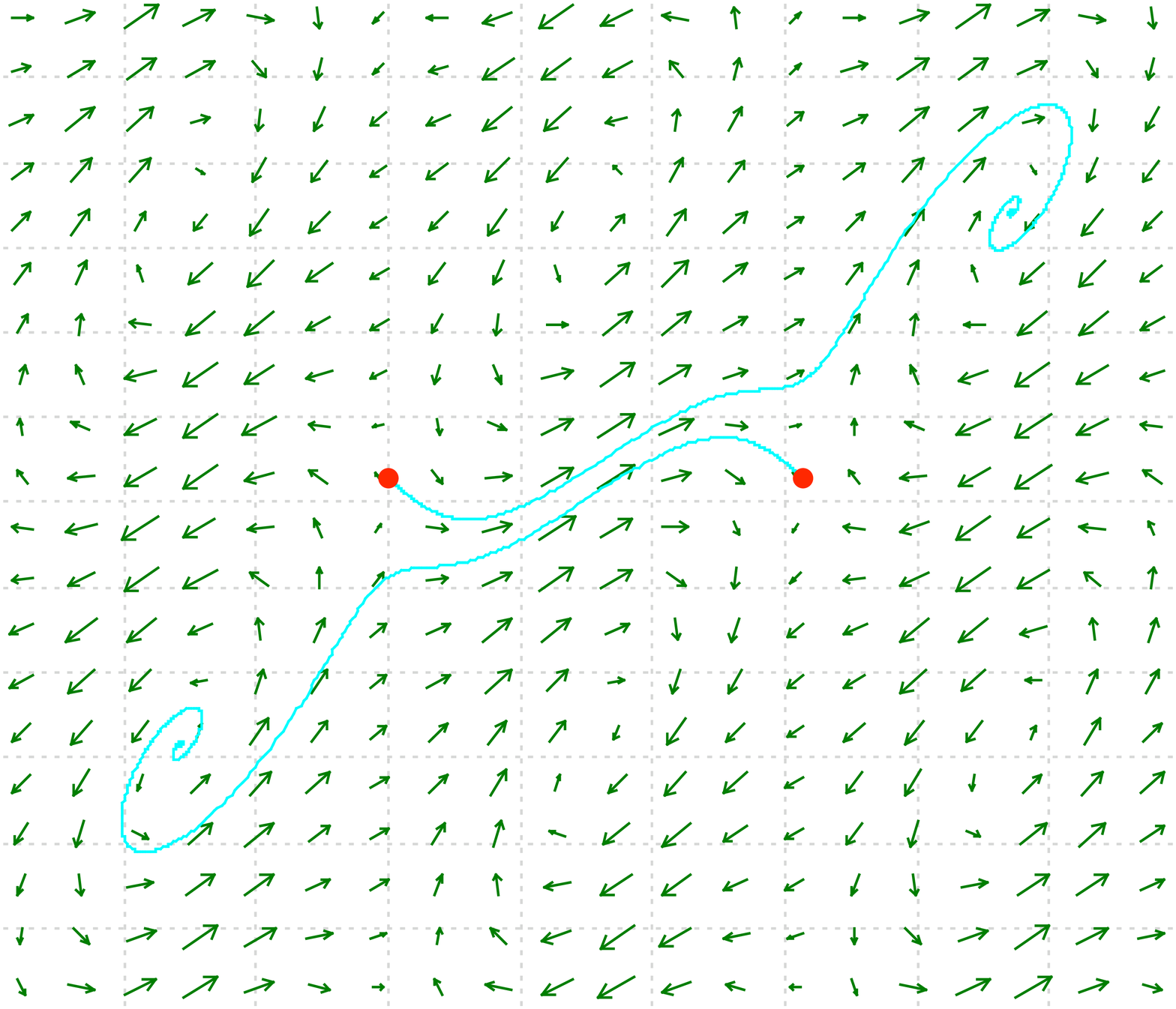}
\caption{the two pictures represent the phase portrait of the dynamical system \eqref{eq:dydthetaphi}
with $U(\vt) = 2-\cos (2\vt)$, when $\alpha = 0.5$ (at left) or $\alpha = 1$ (at right).
We focus our attention on the saddles $(0,\pi)$ and $(\pi,\pi)$ (that satisfy condition \eqref{eq:cond_cos}):
from the mutual positions of the
heteroclinic departing from $(0,\pi)$ and the one ending in $(\pi,\pi)$ we deduce that the
two vector fields are not topologically equivalent.
By structural stability we infer the existence, for some $\bar\alpha\in(0.5,1)$, of a saddle connection
between $(0,\pi)$ and $(\pi,\pi)$.\label{fig:pplane}}
\end{figure}
\begin{lemma}\label{lem:unst_man_control}
Let $(\vt,\vp)$ denote the (unique, apart from time translations) unstable trajectory
emanating from $(\vt^-,\vt^-+\pi)$ with increasing $\vt$.
Then it intersects the line $\vp=\vt+\pi/2$ in an unique point with first coordinate
$\hat \vt^- = \hat \vt^- (\alpha)$. Moreover $\vt$ is strictly increasing on $(\vt^-, \hat \vt^-]$
and on the same interval $\vp = \vp_\alpha(\vt)$ can be expressed as a function of $\vt$.
Finally,
\[
\alpha_1<\alpha_2\quad\text{ implies }\quad \hat\vt^-(\alpha_1)<\hat\vt^-(\alpha_2)
\]
and  $\vp_{\alpha_1}(\vt)<\vp_{\alpha_2}(\vt)$ on $(\vt^-,\hat\vt^-(\alpha_1)]$ (see
also Figure \ref{fig:monoton_manif}).
\end{lemma}
\begin{proof}
To start with we observe that, for any $\alpha \in (0,2)$, the linearized matrix for \eqref{eq:dydthetaphi} at
$(\vt^-, \vt^- + \pi)$ %and $(\vt^+, \vt^+)$
is
\[
J^- = U_{\min}\left( \begin{array}{cc} 2   &   -2 \\   \alpha - \mu^-   &   -\alpha \end{array} \right),
\]
where $\mu^- = U''(\vt^{-})/U_{\min}$.
The eigendirection correspondent to the heteroclinic emanating from $(\vt^-, \vt^- + \pi)$ is
$v^- =(1,v^-_2) = (1,1-\lambda^-_+/2)$, where
$\lambda^-_+ = \left( 2-\alpha + \sqrt{(2-\alpha)^2+8\mu^-} \right)/2$
is the positive eigenvalue of $J^-$; hence
\[
v_2^- = v_2^-(\alpha) = \frac12 + \frac{\alpha}{4} -\frac14 \sqrt{(2-\alpha)^2+8\mu^-}.
\]
On one hand, we have that
\[
\frac{\d}{\d \alpha}v_2^-(\alpha)
= \frac14+\frac{2-\alpha}{4\sqrt{(2-\alpha)^2+8\mu^-}}>0,
\]
implying that, for different values of $\alpha$, the corresponding unstable
trajectories are ordered as claimed near $(\vt^-, \vt^- + \pi)$.
On the other hand, since
$v_2^- <1$, we have that the trajectory is contained in the strip $\pi/2<\vp-\vt<\pi$
for large negative times.

Now recall that, as above, $v(-\infty)=-\sqrt{U_{\min}}$ and that both
$\vt$ and $v$ are strictly increasing whenever $v$ is smaller than $\sqrt{U_{\min}}$.
We deduce that there exists exactly one $\hat\tau$ such that
\[
v(\hat\tau) =0,\text{ or equivalently }\quad \vp(\hat\tau)=\vt(\hat\tau)+\frac{\pi}{2}.
\]
As a consequence the value $\hat\vt^-=\vt(\hat\tau)$ is well defined and, reasoning as in Corollary \ref{coro:as} and in Lemma \ref{lem:stimaalpha}, we can invert
$\vt=\vt(\tau)$ on $(-\infty,\hat \tau]$. We deduce that we can write
\begin{equation}\label{eq:ODE_x_vp}
\vp_\alpha(\vt):=\vp(\tau(\vt)),\text{ where }\quad
\frac{\d\varphi_\alpha}{\d\vartheta} = \frac{\alpha}{2}+
\frac{U'(\vt)}{2U(\vt)}\cotan(\vp_\alpha-\vt)\text{ on }(\vt^-,\hat\vt^-].
\end{equation}
To conclude the proof we have to show that, if $\alpha_1<\alpha_2$, then
$\vp_{\alpha_1}(\vt)<\vp_{\alpha_2}(\vt)$ where they are defined. To this aim,
let by contradiction $\vt^*>\vt^-$ be such that $\vp_{\alpha_1}(\vt)
<\vp_{\alpha_2}(\vt)$ on $(\vt^-,\vt^*)$, and $\vp_{\alpha_1}(\vt^*)=\vp_{\alpha_2}(\vt*)$. But the above differential equation implies
\[
\frac{\d(\varphi_{\alpha_2}-\varphi_{\alpha_1})}{\d\vartheta}(\vt^*) = \frac{\alpha_2-\alpha_1}{2}>0,
\]
a contradiction.
\end{proof}
\begin{figure}[!t]
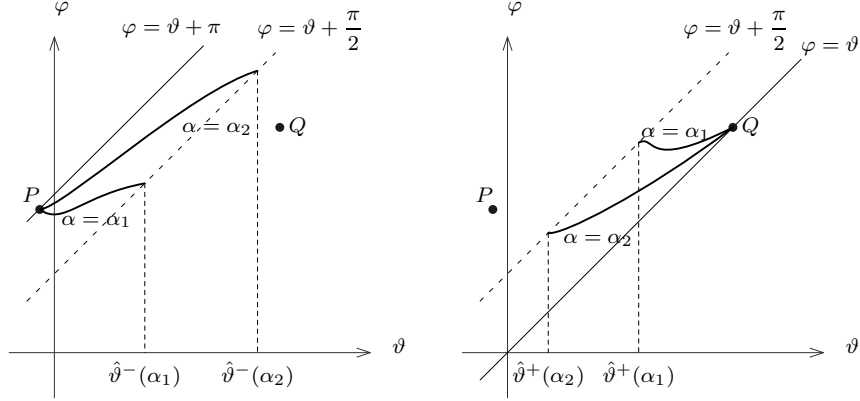

\begin{center}
\begin{tabular}{ccc}
\begin{texdraw}
\drawdim cm  \setunitscale 0.6
%% frecce-assi %%
\linewd 0.01 \setgray 0.2
\move (0 -1) \arrowheadtype t:V \arrowheadsize l:0.3 w:0.2 \avec (0 7)
\htext (0 7.5) {\footnotesize{$\varphi$}}
\move (-1 0) \arrowheadtype t:V \arrowheadsize l:0.3 w:0.2 \avec (7 0)
\htext (7.5 0) {\footnotesize{$\vt$}}
%% rette punti critici %%
\linewd 0.02 \setgray 0
%\move (-.6 -.6) \lvec (6.5 6.5)
%\htext (6.5 6.6) {\footnotesize{$\vp=\vt$}}
\move (-.6 2.9) \lvec (3.3 6.8)
\htext (1.5 7) {\footnotesize{$\vp=\vt+\pi$}}
\lpatt (0.1 0.2) \move (-.6 1.15) \lvec (4.9 6.65)
\htext (4.5 6.7) {\footnotesize{$\displaystyle\vp=\vt+\frac{\pi}{2}$}}
\lpatt()
%% punti di equilibro
\linewd 0.03
\move (-.33 3.17) \fcir f:0.1 r:0.1
\textref h:R v:C \htext (-.33 3.5) {\footnotesize{$P$}}
\move (5 5) \fcir f:0.1 r:0.1
\textref h:L v:C \htext (5.2 5) {\footnotesize{$Q$}}
%% variet� stabile 1
\linewd 0.04
\move (-.33 3.17) \clvec (0.2 2.8)(0.6 3.5)(2 3.75)
\textref h:L v:C \htext (.15 2.9) {\footnotesize{$\alpha=\alpha_1$}}
\textref h:C v:T \htext (2 -0.2) {\footnotesize{$\hat\vt^-(\alpha_1)$}}
\linewd 0.015 \lpatt (0.1 0.1) \move (2 3.75) \lvec (2 0) \lpatt ()
%%% variet� stabile 2
\linewd 0.04
\move (-.33 3.17) \clvec (0.2 3.2)(3 5.8)(4.5 6.25)
\textref h:C v:C \htext (3.6 4.95) {\footnotesize{$\alpha=\alpha_2$}}
\textref h:C v:T \htext (4.5 -0.2) {\footnotesize{$\hat\vt^-(\alpha_2)$}}
\linewd 0.015 \lpatt (0.1 0.1) \move (4.5 6.25) \lvec (4.5 0)
\end{texdraw}

& \hspace{0.1cm} &

\begin{texdraw}
\drawdim cm  \setunitscale 0.6
%% frecce-assi %%
\linewd 0.01 \setgray 0.2
\move (0 -1) \arrowheadtype t:V \arrowheadsize l:0.3 w:0.2 \avec (0 7)
\htext (0 7.5) {\footnotesize{$\varphi$}}
\move (-1 0) \arrowheadtype t:V \arrowheadsize l:0.3 w:0.2 \avec (7 0)
\htext (7.5 0) {\footnotesize{$\vt$}}
%% rette punti critici %%
\linewd 0.02 \setgray 0
\move (-.6 -.6) \lvec (6.5 6.5)
\htext (6.5 6.6) {\footnotesize{$\vp=\vt$}}
%\move (-.6 2.9) \lvec (3.3 6.8)
%\htext (2.7 7) {\footnotesize{$\vp=\vt+\pi$}}
\lpatt (0.1 0.2) \move (-.6 1.15) \lvec (4.9 6.65)
\htext (3.9 6.7) {\footnotesize{$\displaystyle\vp=\vt+\frac{\pi}{2}$}}
\lpatt()
%% punti di equilibro
\linewd 0.03
\move (-.33 3.17) \fcir f:0.1 r:0.1
\textref h:R v:C \htext (-.33 3.5) {\footnotesize{$P$}}
\move (5 5) \fcir f:0.1 r:0.1
\textref h:L v:C \htext (5.2 5) {\footnotesize{$Q$}}
%% variet� stabile 1
\linewd 0.04
\move (2.9 4.65) \clvec (3.3 4.85)(3 4)(5 5)
\textref h:R v:C \htext (4.5 4.85) {\footnotesize{$\alpha=\alpha_1$}}
\textref h:C v:T \htext (2.9 -0.2) {\footnotesize{$\hat\vt^+(\alpha_1)$}}
\linewd 0.015 \lpatt (0.1 0.1) \move (2.9 4.65) \lvec (2.9 0) \lpatt ()
%% variet� stabile 2
\linewd 0.04
\move (0.9 2.65) \clvec (1.2 2.6)(3 3.5)(5 5)
\textref h:C v:C \htext (2 2.5) {\footnotesize{$\alpha=\alpha_2$}}
\textref h:C v:T \htext (.9 -0.2) {\footnotesize{$\hat\vt^+(\alpha_2)$}}
\linewd 0.015 \lpatt (0.1 0.1) \move (0.9 2.65) \lvec (0.9 0)
\end{texdraw}
\end{tabular}
\end{center}
\caption{the unstable (resp. stable) manifold emanating from $P=(\vt^-,\vt^-+\pi)$
(resp. entering in $Q=(\vt^+,\vt^+)$), and its dependence on $\alpha$, according
to Lemma \ref{lem:unst_man_control} (resp. Lemma \ref{lem:st_man_control}).
Here $\alpha_1<\alpha_2$.\label{fig:monoton_manif}}
\end{figure}
Arguing exactly as above one can prove analogous properties for the stable manifolds.
\begin{lemma}\label{lem:st_man_control}
Let $(\vt,\vp)$ denote the (unique, apart from time translations) stable trajectory
entering in $(\vt^+,\vt^+)$ with increasing $\vt$.
Then it intersects the line $\vp=\vt+\pi/2$ in an unique point with first coordinate
$\hat \vt^+ = \hat \vt^+ (\alpha)$. Moreover $\vt$ is strictly increasing on $[\hat\vt^+, \vt^-)$
and on the same interval $\vp = \vp_\alpha(\vt)$ can be expressed as a function of $\vt$.
Finally,
\[
\alpha_1<\alpha_2\quad\text{ implies }\quad \hat\vt^+(\alpha_1)>\hat\vt^+(\alpha_2)
\]
and  $\vp_{\alpha_1}(\vt)>\vp_{\alpha_2}(\vt)$ on $[\hat\vt^+(\alpha_1),\vt^+)$.
\end{lemma}
By uniqueness, the above unstable/stable trajectories can not be crossed by any other orbit.
To be more precise, we have the following.
\begin{corollary}\label{coro:barriere}
Let
\[
\vt^*\in[\vt^-,\vt^+]\text{ such that } U'(\vt^*)=0
\]
be \emph{any} central configuration, and $\gamma$ be a trajectory of system \eqref{eq:dydthetaphi} emanating from $(\vt^*,\vt^*+\pi)$ and intersecting the
set
\[
 \Sigma:=\left\{(\vt,\vp):\, \vt^-\leq\vt\leq\vt^+,\,\vt+\frac{\pi}{2}\leq\vp\leq\vt
 +\frac{3\pi}{2}\right\}.
\]
Then, if $\gamma$ exits from $\Sigma$, it must cross either the union of the segments
\[
 \left\{\hat\vt^-(\alpha)\leq\vt\leq\vt^+,\,\vp=\vt+\frac{\pi}{2}\right\},
 \quad
 \left\{\vt^-\leq\vt\leq\hat\vt^+(\alpha),\,\vp=\vt+\frac{3\pi}{2}\right\},
 \]
or the vertical lines $\vt=\vt^\pm$.
Analogously, for a trajectory asymptotic (in the future) to $(\vt^*,\vt^*)$,
the entering set in
\[
\Sigma':=\left\{(\vt,\vp):\, \vt^-\leq\vt\leq\vt^+,\,\vt-\frac{\pi}{2}\leq\vp\leq\vt
+\frac{\pi}{2}\right\}
\]
is the union of the segments
 \[
 \left\{\hat\vt^+(\alpha)\leq\vt\leq\vt^+,\,\vp=\vt-\frac{\pi}{2}\right\},
 \quad
 \left\{\vt^-\leq\vt\leq\hat\vt^-(\alpha),\,\vp=\vt+\frac{\pi}{2}\right\},
 \]
and of the vertical lines $\vt=\vt^\pm$.
\end{corollary}
\begin{proof}
We prove only the first part. If $\vt^*=\vt^-$ then $\gamma\equiv\gamma_1$, the unique unstable trajectory emanating from the corresponding saddle point with $\vt$ increasing
considered in Lemma \ref{lem:unst_man_control}; but then it exits from $\Sigma$ through
the point $(\hat\vt^-(\alpha),\hat\vt^-(\alpha)+\pi/2)$. In the same way, if
$\vt^*=\vt^+$ then $\gamma\equiv\gamma_2$, the unique unstable trajectory emanating
from the corresponding saddle point with $\vt$ decreasing (recall that, if
$(\vt(\tau),\vp(\tau))$ solves \eqref{eq:dydthetaphi}, then also $(\vt(-\tau),\vp(-\tau)
+\pi)$ does); in such a case the exit point is $(\hat\vt^+(\alpha),\hat\vt^+(\alpha)+
3\pi/2)$. Finally, if $\vt^-<\vt^*<\vt^+$, then $\gamma$ must lie above $\gamma_1$ and
below $\gamma_2$, and the assertion follows.
\end{proof}
The angles $\hat \vt^\pm (\alpha)$ defined above represent the (oriented) parabolic
\emph{apsidal angles} swept by the parabolic arc from the infinity up to the
pericenter. As a consequence of the previous arguments, the appearance of a parabolic
trajectory associated with the asymptotic directions  $(\vt^-,\vt^+)$, or,
equivalently, the existence of a
heteroclinic connection between $(\vt^-, \vt^- + \pi)$ and $(\vt^+, \vt^+)$ can be
expressed in terms of the corresponding apsidal angles. Summing up, we have proved the following.
\begin{proposition}\label{propo:unique_baralfa}
Let $U \in \Uh$, $\vt^- < \vt^+ \in \Theta$, and the
monotone functions $\hat\vt^-(\alpha)$, $\hat\vt^+(\alpha)$ be defined as in Lemmata
\ref{lem:unst_man_control}, \ref{lem:st_man_control}, respectively. Then system \eqref{eq:dydthetaphi} admits a heteroclinic connection between $(\vt^-, \vt^- + \pi)$ and $(\vt^+, \vt^+)$ for some value $\alpha=\bar\alpha \in (0,2)$ if and only if
\[
\hat\vt^-(\bar\alpha)=\hat\vt^+(\bar\alpha).
\]
In particular, if such a value exists, then it is unique.
\end{proposition}
The function $\bar\alpha$ can be extended to all the possible triplets $U \in \Uh_{
\vt_1 \vt_2}$ and $\vt^- \;,\vt^+ \in \Theta_{\vt_1 \vt_2}$  as follows:
\begin{definition}\label{defi:baralfa}
For any triplet $U \in \Uh$, $\vt^- < \vt^+ \in \Theta$,
we define the function
\[
\bar\alpha(\vt^-,\vt^+,U)=\inf\left\{\alpha\in(0,2):\,\hat\vt^-(\alpha)>
\hat\vt^+(\alpha)\right\}
\]
If $\vt^- >  \vt^+$ we define $\bar\alpha(\vt^-,\vt^+,U)=\bar\alpha(\vt^+,\vt^-,U)$.
\end{definition}
In this way, the previous proposition proves the first point of Theorem \ref{theo:main}.

As a final remark, let us notice that the apsidal angles defined above, and
the corresponding stable/unstable trajectories, act as a ``barrier'' for any
heteroclinic traveling in the strip $\vt^-\leq\vt\leq\vt^+$ and
corresponding to a (not necessarily minimal) parabolic trajectory. Such kind of arguments
will turn out to be useful in the proof of Theorems \ref{theo:bolza} and \ref{theo:gordon}.
\begin{proposition}\label{propo:no_para_in_strip}
Let $U \in \Uh$, $\vt^-, \vt^+ \in \Theta$, and
let us assume that
\[
\alpha>\bar\alpha(\vt^-,\vt^+,U).
\]
Then $(U,\alpha)$ does not admit any (not necessarily minimal) parabolic trajectory
completely contained in the sector $[\vt^-,\vt^+]$.
\end{proposition}
\begin{proof}
By Theorem \ref{thm:devaney} (and in particular condition \eqref{eq:cond_cos})
such a parabolic trajectory $x=x(t)$ would correspond to an heteroclinic connection
for system \eqref{eq:dydthetaphi}, joining an equilibrium (say)
$(\vt^*,\vt^*+\pi)$ to another one $(\vt^{**}+2h\pi,\vt^{**}+2h\pi)$,
with $h$ integer. We want to prove that such a trajectory, under the above
assumptions, can not be completely contained in the strip $[\vt^-,\vt^+]\times\RR$.

To start with, we observe that $h$ must be equal to either $0$ or $1$. Indeed, the
function $v(\tau)$ is non-decreasing along any trajectory, and $v=0$ whenever
$\vp=\vt+\pi/2 + k\pi$, $k$ integer. W.l.o.g we can assume $h=0$, so that the
trajectory we are considering joins $(\vt^*,\vt^*+\pi)$ to $(\vt^{**},\vt^{**})$. Let
us assume by contradiction that it is completely contained in the strip
$[\vt^-,\vt^+]\times\RR$; but then, using the notations of Corollary
\ref{coro:barriere}, it must both exit $\Sigma$ and enter $\Sigma'$, across
a single point belonging to the line $\vp=\vt+\pi/2$ and the strip.
This immediately provides a contradiction with the selfsame corollary, since
\[
\alpha > \bar\alpha
\qquad\Longrightarrow\qquad\hat\vt^-(\bar\alpha)>\hat\vt^+(\bar\alpha).\qedhere
\]
\end{proof}
%=======================================
\section{Minimality Properties near Equilibria}\label{sec:min_prop}
%=======================================

The purpose of this section is to develop a first investigation about the
minimality properties of zero energy solutions of \eqref{eq:dynsys}
with respect to the Maupertuis' functional
\[
J(x) = J([a,b];x):=
\int_{a}^{b} \frac12 |\dot x(t)|^2  \, \dt \cdot  \int_{a}^{b} V(x(t)) \, \dt,
\]
where $V=(U,\alpha) \in \Vh$. Indeed let us recall that
\begin{multline*}
\min \left\{ \action([a',b'];y) : \, a'<b', y \in H^1(a',b') \, + \, \text{further conditions} \right\}
= \\
\min \left\{ \sqrt{2J([a,b];x)} : x \in H^1(a,b) \, + \, \text{same conditions} \right\}
\end{multline*}
\emph{for every (fixed) $a<b$}, indeed $J$ is invariant under reparameterizations (see \cite{AC-Z}).
As a consequence, every parabolic trajectory is a critical point of $J$,
at least when restricted on suitably small bounded intervals.

In particular, we want to evaluate the second differential of $J$ along zero-energy critical points.
In order to do this, we first perform a change of time-scale essentially equivalent to
the Devaney's one we exploited in Section \ref{sec:devaney}. In polar coordinates $J$ reads as
\[
J(r,\vt) =
   \int_{a}^{b} \frac12 \left[ \dot r^2(t) + r^2(t)\dot\vt^2(t) \right]\dt\cdot
   \int_{a}^{b}  \frac{U(\vt(t))}{r^{\alpha}(t)}\dt;
\]
introducing the time-variable
\[
\tau = \tau(t) = \int_a^t r^{-(2+\al)/2}(\xi)\,\d\xi
\]
we obtain (noting with a prime `` $'$ '' the derivative with respect to $\tau$ )
\[
J(r,\vt) =
   \int_{a}^{\tau^*} \frac12 \left[\left(r^{-(2+\alpha)/4}r'\right)^2
                                 + \left(r^{(2-\alpha)/4}\vt' \right)^2 \right]\dtau
   \cdot\int_{a}^{\tau^*} r^{(2-\alpha)/2}U(\vt)\dtau
\]
where $r$ and $\vt$ depends now on $\tau$, and
\[
\tau^* = \int_a^{b} r^{-\frac{2+\al}{2}}dt.
\]
We introduce the change of variables
\begin{equation*}
%\label{rho}
\rho = r^{\frac{2-\al}{4}}, \quad \rho' = \frac{2-\al}{4}
r^{-\frac{2+\al}{4}} r'
\end{equation*}
in order to obtain the Maupertuis' functional depending on $(\rho,\vt)$, i.e.
\begin{equation*}
J (\rho,\vt) = F(\rho,\vt)\,G(\rho,\vt),
%\label{mau_rho_tau}
\end{equation*}
where
\[
F(\rho,\vt) = \int_0^{\tau^*} \frac{8}{(2-\alpha)^2} (\rho')^2 + \frac{1}{2} (\rho\vt')^2 \dtau, \quad
G(\rho,\vt) = \int_0^{\tau^*} \rho^2 U(\vt)\dtau.
\]
The energy relation \eqref{eq:dynsys2} written in terms of $\tau$, $\rho$
and $\vt$ yields $F(\rho,\vt) = G(\rho,\vt)$.
Let now $(\rho,\vt)$ be a critical point of $J$, then
\[
\d J(\rho,\vt) = \d F(\rho,\vt)\,G(\rho,\vt) + F(\rho,\vt)\,\d G(\rho,\vt) = 0;
\]
from the energy relation we then deduce that if $(\rho, \vt)$ is a zero-energy critical point for $J$ then $\d F(\rho,\vt) = -\d G(\rho,\vt)$; as a consequence we can write
\[
\d^2J(\rho,\vt) = G(\rho,\vt)\, \left[\d^2F(\rho,\vt) + \d^2 G(\rho,\vt)\right] -2 \left[ \d G(\rho,\vt)\right]^2.
\]
More explicitly, given  a compactly supported variation $(\lambda,\xi)$ and
a zero-energy critical point for $J$, $(\rho, \vt)$, we have that
\[
\begin{split}
\d^2F(\rho,\vt)[(\lambda,\xi),(\lambda,\xi)] &=
\int_0^{\tau^*} \frac{16}{(2-\alpha)^2} (\lambda')^2 + (\rho\xi')^2
                 + 4\vt' \xi' \rho \lambda +(\vt')^2\lambda^2\dtau,  \\
\d  G(\rho,\vt)(\lambda,\xi) &=  \int_0^{\tau^*} 2\rho\lambda U(\vt) + \rho^2 U'(\vt)\xi \dtau, \\
\d^2G(\rho,\vt)[(\lambda,\xi),(\lambda,\xi)] &= \int_0^{\tau^*} 2 \lambda^2U(\vt) +
                                                4\lambda \rho U'(\vt)\xi + \rho^2 U''(\vt)\xi^2 \dtau.
\end{split}
\]
In the rest of the paper we will prove that trajectories asymptotic to minimal
central configurations are indeed, at least locally, minimizers for $J$.
The main result of this section concerns the non-minimality of trajectories
which are asymptotic to ``sufficiently'' non-minimal central configurations.
\begin{proposition}\label{propo:maximal_not_minimal}
Let $\bar\vt$ be such that $U'(\bar\vt)=0$, and let $(\rho,\vt)$ be any critical
point of $J$, defined for $\tau\in[0,+\infty)$, such that $\vt(\tau)\to\bar\vt$
as $\tau\to+\infty$. Finally let $\alpha$ be such that
\begin{equation}\label{eq:too_strict}
U''\left(\bar\vt\right)<-\frac{(2-\alpha)^2}{8} U\left(\bar\vt\right).
\end{equation}
Then, for $a'<b'$ sufficiently large, $(\rho,\vt)$ restricted to $(a',b')$ is neither
a minimum for $\action$, nor for $J$.
\end{proposition}
\begin{proof}
We prove the result for the Maupertuis' functional $J$, indeed the computations for
the action are similar but simpler (recall that, under the above notations,
$\action (\rho,\vt) = F(\rho,\vt)+G(\rho,\vt)$). More precisely, we are going to provide a
compactly supported variation $(0,\xi)$ along which $d^2J(\rho,\vt)$ will result
negative. By the above calculations we have
\begin{multline*}
\d^2J(\rho,\vt)[(0,\xi),(0,\xi)] = \\
\int_{a'}^{b'} \rho^2 U(\vt) \, \dtau \cdot
\int_{a'}^{b'} \rho^{2} \left[  (\xi')^2 + U''(\vt) \xi^2 \right] \dtau
-2\left( \int_{a'}^{b'} \rho^2 U'(\vt)\xi \, \dtau,\right)^2\\
\leq C\int_{a'}^{b'} \rho^{2} \left[  (\xi')^2 + (\mu+\ep) \xi^2 \right] \dtau,
\end{multline*}
where $C>0$, $\ep>0$ is small, $a'<b'$ are large and $\mu:=U''(\bar\vt)$.

Now, we claim that the solutions of the linear equation
\begin{equation}\label{eq:yetanotherstupidequation}
(\rho^2 \xi')'=(\mu+2\ep)\rho^2\xi
\end{equation}
have infinitely many zeroes for $\tau$ large; as a consequence, choosing $a',b'$ to be two of such zeroes, testing with $\xi$ and integrating by parts, one would obtain
\[
\int_{a'}^{b'} \rho^{2} \left[  (\xi')^2 + (\mu+\ep) \xi^2 \right] \dtau = -\ep\int_{a'}^{b'} \rho^{2} \xi^2  \dtau<0,
\]
providing the desired result.

In order to establish the oscillatory nature of equation \eqref{eq:yetanotherstupidequation} we will apply Sturm comparison principle.
First of all, by combining the Euler-Lagrange
equation for $\rho$
\[
\frac{16}{(2-\alpha)^2} \rho'' = (\vt')^2 \rho + 2\rho U(\vt),
\]
and the zero-energy relation
\[
\frac{8}{(2-\alpha)^2} (\rho')^2 +\frac12 (\vt')^2 \rho^2 = \rho^2 U(\vt),
\]
we have that the function
\[
p(\tau):=\frac{\rho'(\tau)}{\rho(\tau)}\quad\text{ satisfies }\quad
p'= - 2p^2+\frac{(2-\alpha)^2}{4}U(\vt)
\]
on $[0,+\infty)$. But then, since $\vt(\tau)\to\bar\vt$ as $\tau\to+\infty$,
by elementary comparison we easily obtain
\[
\lim_{\tau\to+\infty} \frac{\rho'(\tau)}{\rho(\tau)} = \sqrt{\frac{(2-\alpha)^2}{8}
U(\bar\vt)}=:\gamma.
\]
We finally infer that, for some constant $k$, and for $\tau$ large, there holds
$\rho(\tau)< k e^{(\gamma+\ep)\tau}$. But then Sturm comparison principle applies
to \eqref{eq:yetanotherstupidequation} and to
\[
(k^2 e^{2(\gamma+\ep)\tau} \xi')'=(\mu+2\ep)k^2 e^{2(\gamma+\ep)\tau}\xi,
\]
yielding that every nodal interval of the second equation contains (at least) one zero
of the first one; to conclude we observe that this last equation writes
\[
\xi''+2(\gamma+\ep)\xi'-(\mu+2\ep)\xi=0,
\]
which is oscillatory if and only if, for some $\ep>0$, there holds
$(\gamma+\ep)^2+(\mu+2\ep)<0$, i.e. if and only if
\[
\mu<-\gamma^2.\qedhere
\]
\end{proof}
\begin{corollary}\label{coro:maximal_not_minimal}
Let $\bar\vt$ be such that $U'(\bar\vt)=0$, and let $x$, defined for $t\in[0,+\infty)$,
be a solution of \eqref{eq:dynsys}, \eqref{eq:dynsys2},
such that $x(t)/|x(t)| \to (\cos \bar\vt, \sin \bar\vt)$ as $t\to+\infty$.
Finally let $\alpha$ satisfy condition \eqref{eq:too_strict}.
Then, $x$ can neither be a free-time Morse minimizer, nor a fixed-time one.
\end{corollary}
Let us mention that this result completely agrees with the one proved, in the complementary case of collision trajectories, in \cite{BaSe}; on the other hand, quite surprisingly, it is not clear wether trajectories
corresponding to ``not too-strict'' maxima for $U$ (i.e. maxima such that
$-\gamma^2<U''(\bar\vt)<0$) may be minimizers for $J$.

%===============================
\section{Constrained Minimizers}\label{sec:sector}
%===============================

In this section we prove Theorem \ref{theo:main} in the case in which
$\vt^-,\vt^+\in\Theta_{\vt_1 \vt_2}$ are such that
\[
\pi < \vt^+ - \vt^- \leq 2\pi.
\]
By time reversibility, also the case $-2\pi \leq \vt^+ - \vt^- <-\pi$ will follow.
In such situation the results in \cite{BTV} apply almost straightforwardly; we summarize them here, making explicit the minor changes we need in the present situation.

The main idea is that, since parabolic minimizers exist only for special values of
$\alpha$, one first introduces more general objects which, on the contrary, exist for every $\alpha$.
\begin{definition}\label{defi:constr_Morse_min}
We say that $x =(r,\vt)\in H^1_{\mathrm{loc}}(\RR)$ is a \emph{constrained Morse minimizer} if
\begin{itemize}
 \item $\min_t r(t)=1$;
 \item $r(t)\to+\infty$, $\vt(t)\to\vt^\pm$ as $t\to\pm\infty$;
 \item for every $t_1<t_2$, $t_1'<t_2'$, and %$z=(\rho,\zeta)\in H^1(t_1',t_2')$,
$z\in H^1(t_1',t_2')$, there holds
\begin{multline*}
z(t_i')=x(t_i),\,i=1,2,\, \min_{[t_1',t_2']}|z|=\min_{[t_1,t_2]}r \\
\implies\quad \action([t_1,t_2];x)\leq\action([t_1',t_2'];z).
\end{multline*}
\end{itemize}
We denote with $\morse=\morse(U,\alpha)$ the set of constrained Morse minimizers.
\end{definition}
As for Definition \ref{defi:Morse_min}, also the previous definition makes sense
for any pair of central configurations, not necessarily for minimal ones. From this
point of view, Proposition \ref{propo:maximal_not_minimal} provides a necessary
condition for $\morse$ to be non-empty, in the case of non-minimal central
configurations. In any case, when not explicitly remarked, we will always refer
to constrained minimizers between minimal asymptotic configurations.

The following two lemmas describe the main properties of constrained
minimizers; they are a direct consequence of the theory developed in \cite{BTV},
Sections 5 and 6.
\begin{lemma}\label{lem:descrizMorse}
For every $\alpha\in(0,2)$ the set $\morse$ is not empty. If
$x=(r,\vt) \in\morse$ then (up to a time translation)
there exist $t_* \leq 0 \leq t_{**}$ such that:
\begin{enumerate}
\item $r(t) = 1$ if and only if $t \in [t_{*},t_{**}]$,
$\dot{r}(t)< 0$ (resp. $>0$) if and only if $t <t_*$ (resp. $t > t_{**}$);
\item $x$ satisfies \eqref{eq:dynsys} for every $t \not\in [t_{*},t_{**}]$ and
\eqref{eq:dynsys2} for every $t$;
\item one of the following alternatives hold:
\begin{enumerate}
\item $t_*<t_{**}$, $x$ is $\cont^1$ for every $t$, $\dot r\equiv 0$ in $[t_{*},t_{**}]$;
\item $t_*=t_{**}=0$ and $x$ is $\cont^1$ for every $t$;
\item $t_*=t_{**}=0$ and $\dot x$ has a jump discontinuity at $0$, with
\[
-\dot r(0^-)=\dot r(0^+)>0,\qquad \dot\vt(0^-)=\dot\vt(0^+).
\]
\end{enumerate}
\end{enumerate}
\end{lemma}

\begin{definition}
In view of the previous lemma, for any $x=(r,\vt)\in\morse$
we define its \emph{(angular) position and velocity jumps} respectively as
\[
\pjump(x):= |\vt(t_{**}) - \vt(t_{*})|,\qquad
\vjump(x):= |\dot r(t_{**}^+) - \dot r(t_{*}^-)|
\]
(in particular they can not be both different from 0, while they are both 0 if and only if
alternative (b) above holds).
\end{definition}

\begin{lemma}\label{lem:exist_morse}
Let $0<\alpha_1<\alpha_2<1$ and let us assume that there
exists $x_i\in\morse(U,\alpha_i)$, $i=1,2$, such that
\[
\pjump(x_1)>0\quad\text{ and }\quad\vjump(x_2)>0.
\]
Then there exist $\bar\alpha\in(\alpha_1,\alpha_2)$ and
$\bar x\in\morse(U,\bar\alpha)$ such that
\[
\pjump(\bar x)=\vjump(\bar x)=0\quad\text{ and }\bar x\text{ is a corresponding free
Morse minimizer}.
\]
\end{lemma}
In the planar case, the general theory we have recalled above can be complemented
using the results about the Devaney's system that we obtained in Section \ref{sec:devaney}.
\begin{remark}\label{rem:anotherfuckingkeyremark}
Let $x=(r,\vt)\in\morse$ and $t_{*} \leq 0 \leq t_{**}$ be as in Lemma \ref{lem:descrizMorse}.
Via the variable and time changes introduced in Section \ref{sec:devaney},
we can define $\tau_* \leq 0 \leq \tau_{**}$ in order to obtain
that $x|_{\{t<t_{*}\}}$ corresponds, in the phase plane of system \eqref{eq:dydthetaphi},
to a part of the unstable trajectory emanating from $(\vt^-,\vt^-+\pi)$,
with $\vt$ increasing (and $\tau <\tau_*$); moreover $\vt^-<\vt<\hat\vt^-(\alpha)$
along the trajectory (both this facts descend from the fact that, for $t<t_{*}$,
$\dot r$ is negative, and thus also $r'$ is for $\tau <\tau_*$).
Analogously, $x|_{\{t>t_{**}\}}$ corresponds to a part of the stable trajectory
entering in $(\vt^+,\vt^+)$ (with $\tau>\tau_{**}$), and $\hat\vt^+(\alpha)<\vt<\vt^+$.

Finally, for $\tau\in(\tau_*,\tau_{**})$ (whenever such interval is non empty),
$(\vt,\vp)$ lies in a 1-to-1 way on the line of equation $\vp=\vt+\pi/2$. In particular, $(\vt,\vp)$
is completely contained in the strip
\[
\left\{(\vt,\vp):\, \vt^-<\vt<\vt^+,\,\vt<\vp<\vt+\pi\right\}.
\]
\end{remark}
Taking into account Lemmata \ref{lem:unst_man_control} and \ref{lem:st_man_control} it
is possible to give a full characterization of constrained Morse minimizers in terms
of the functions $\hat\vt^\pm(\alpha)$ there defined.
\begin{proposition}\label{propo:classification_constrained}
Let $U\in\Uh$, $\alpha\in(0,2)$, $\vt^\pm$ as above. Then the corresponding
constrained Morse minimizer $x_\alpha$ is unique (up to time translations) and (see
also Figure \ref{fig:pvjump_devaney})
\begin{itemize}
\item $\pjump(x_\alpha)>0$ if and only if $\hat\vt^-(\alpha)<\hat\vt^+(\alpha)$;
\item $\vjump(x_\alpha)>0$ if and only if $\hat\vt^-(\alpha)>\hat\vt^+(\alpha)$;
\item $\pjump(x_\alpha)=\vjump(x_\alpha)=0$ if and only if
$\hat\vt^-(\alpha)=\hat\vt^+(\alpha)$.
\end{itemize}
\end{proposition}
\begin{figure}[!t]
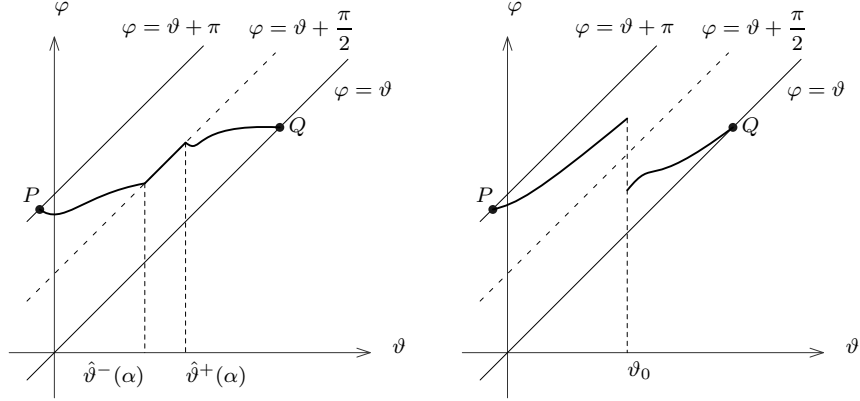

\begin{center}
\begin{tabular}{ccc}
\begin{texdraw}
\drawdim cm  \setunitscale 0.6
%% frecce-assi %%
\linewd 0.01 \setgray 0.2
\move (0 -1) \arrowheadtype t:V \arrowheadsize l:0.3 w:0.2 \avec (0 7)
\htext (0 7.5) {\footnotesize{$\varphi$}}
\move (-1 0) \arrowheadtype t:V \arrowheadsize l:0.3 w:0.2 \avec (7 0)
\htext (7.5 0) {\footnotesize{$\vt$}}
%% rette punti critici %%
\linewd 0.02 \setgray 0
\move (-.6 -.6) \lvec (6.5 6.5)
\htext (6.2 5.6) {\footnotesize{$\vp=\vt$}}
\move (-.6 2.9) \lvec (3.3 6.8)
\htext (1.5 7) {\footnotesize{$\vp=\vt+\pi$}}
\lpatt (0.1 0.2) \move (-.6 1.15) \lvec (4.9 6.65)
\htext (4.3 6.7) {\footnotesize{$\displaystyle\vp=\vt+\frac{\pi}{2}$}}
\lpatt()
%% punti di equilibro
\linewd 0.03
\move (-.33 3.17) \fcir f:0.1 r:0.1
\textref h:R v:C \htext (-.33 3.5) {\footnotesize{$P$}}
\move (5 5) \fcir f:0.1 r:0.1
\textref h:L v:C \htext (5.2 5) {\footnotesize{$Q$}}
%% variet� stabile 1
\linewd 0.04
\move (-.33 3.17) \clvec (0.2 2.8)(0.6 3.5)(2 3.75)
%\textref h:L v:C \htext (.15 2.9) {\footnotesize{$\alpha=\alpha_1$}}
\textref h:R v:T \htext (2 -0.2) {\footnotesize{$\hat\vt^-(\alpha)$}}
\linewd 0.015 \lpatt (0.1 0.1) \move (2 3.75) \lvec (2 0) \lpatt ()
%% variet� stabile 1
\linewd 0.04
\move (2.9 4.65) \clvec (3.3 4.35)(3 5.1)(5 5)
%\textref h:R v:C \htext (4.5 4.85) {\footnotesize{$\alpha=\alpha_1$}}
\textref h:L v:T \htext (2.9 -0.2) {\footnotesize{$\hat\vt^+(\alpha)$}}
\linewd 0.015 \lpatt (0.1 0.1) \move (2.9 4.65) \lvec (2.9 0) \lpatt ()
%archetto
\linewd 0.04
\move (2 3.75) \lvec (2.9 4.65)
\end{texdraw}

& \hspace{0.1cm} &

\begin{texdraw}
\drawdim cm  \setunitscale 0.6
%% frecce-assi %%
\linewd 0.01 \setgray 0.2
\move (0 -1) \arrowheadtype t:V \arrowheadsize l:0.3 w:0.2 \avec (0 7)
\htext (0 7.5) {\footnotesize{$\varphi$}}
\move (-1 0) \arrowheadtype t:V \arrowheadsize l:0.3 w:0.2 \avec (7 0)
\htext (7.5 0) {\footnotesize{$\vt$}}
%% rette punti critici %%
\linewd 0.02 \setgray 0
\move (-.6 -.6) \lvec (6.5 6.5)
\htext (6.2 5.6) {\footnotesize{$\vp=\vt$}}
\move (-.6 2.9) \lvec (3.3 6.8)
\htext (1.5 7) {\footnotesize{$\vp=\vt+\pi$}}
\lpatt (0.1 0.2) \move (-.6 1.15) \lvec (4.9 6.65)
\htext (4.3 6.7) {\footnotesize{$\displaystyle\vp=\vt+\frac{\pi}{2}$}}
\lpatt()
%% punti di equilibro
\linewd 0.03
\move (-.33 3.17) \fcir f:0.1 r:0.1
\textref h:R v:C \htext (-.33 3.5) {\footnotesize{$P$}}
\move (5 5) \fcir f:0.1 r:0.1
\textref h:L v:C \htext (5.2 5) {\footnotesize{$Q$}}
%% variet� stabile 1
\linewd 0.04
\move (-.33 3.17) \clvec (0.2 3.3)(0.6 3.5)(2.65 5.2)
%% variet� stabile 1
\linewd 0.04
\move (2.65 3.6) \clvec (3.3 4.4)(3 3.5)(5 5)
\textref h:L v:T \htext (2.65 -0.2) {\footnotesize{$\vt_0$}}
\linewd 0.015 \lpatt (0.1 0.1) \move (2.65 5.2) \lvec (2.65 0) \lpatt ()
\end{texdraw}
\end{tabular}
\end{center}
\caption{on the left, a position-jumping constrained minimizer between
$P=(\vt^-,\vt^-+\pi)$ and $Q=(\vt^+,\vt^+)$, with $\pjump(x_\alpha)=
\hat\vt^+(\alpha)-\hat\vt^+(\alpha)$. On the right, a velocity-jumping one;
in such a case, the jump discontinuity
is symmetric with respect to $(\vt_0,\vt_0+\pi/2)$ (see equation
\eqref{eq:vp_jump}).\label{fig:pvjump_devaney}}
\end{figure}
\begin{proof}
Let $x_\alpha$ be \emph{any} element of $\morse(U,\alpha)$ and let us denote with
$(\vt,\vp)$ the corresponding arc in the Devaney's plane. Moreover, let
$\tau_*\leq0\leq\tau_{**}$ be the values of $\tau$ corresponding to $t_{*}$,
$t_{**}$, respectively.

We start by assuming that $\vjump(x_\alpha)>0$. This means that $\dot r$ never
vanishes, it is not defined in $t_*=t_{**}=0$, and $-\dot r(0^-)=\dot r(0^+)>0$.
Since $x_\alpha$, and hence $\vt$, are continuous, we obtain that $\vp$ must be
discontinuous. More precisely, letting $\vt_0:=\vt(0)$ and recalling system
\eqref{eq:syscomplete}, we have that
\[
-2 U(\vt_0) \cos (\vp(0^-)-\vt_0)= -r'(0^-) = r'(0^+) =2 U(\vt_0) \cos (\vp(0^+)-\vt_0),
\]
which implies (recall also Remark \ref{rem:anotherfuckingkeyremark})
\begin{equation}\label{eq:vp_jump}
\dfrac{\vp(0^+)+\vp(0^-)}{2}=\vt_0+\frac{\pi}{2}.
\end{equation}
On the other hand, for $\tau$ negative (resp. positive) we have that $\vp$ must be
greater (resp. lower) than $\vt+\pi/2$. Recalling Lemmas \ref{lem:unst_man_control},
\ref{lem:st_man_control}, we deduce that $\hat\vt^-(\alpha)>\hat\vt^+(\alpha)$.
Let us now show that, for every $\alpha$ satisfying this last condition, there exists
exactly one $\vt_0 \in (\hat\vt^+(\alpha),\hat\vt^-(\alpha))$ such that condition
\eqref{eq:vp_jump} holds; this, together with the fact that $r(0)=1$, will imply
uniqueness for the velocity-jumping constrained minimizer.
Thanks to Lemmas \ref{lem:unst_man_control}, \ref{lem:st_man_control} we have that, for $\vt\in
(\hat\vt^+(\alpha),\hat\vt^-(\alpha))$, both the unstable manifold $\vp=\vp_-(\vt)$
and the stable one $\vp=\vp_+(\vt)$ are well defined as functions of $\vt$, and that
they both satisfy equation \eqref{eq:ODE_x_vp}, i.e.
\[
\frac{\d\varphi_\pm}{\d\vartheta} = \frac{\alpha}{2}+
\frac{U'(\vt)}{2U(\vt)}\cotan(\vp_\pm-\vt).
\]
Let us define the (smooth) auxiliary function $\psi(\vt):=\vp_+(\vt)+\vp_-(\vt)-2\vt-\pi$.
Then condition \eqref{eq:vp_jump} is equivalent to $\psi(\vt_0)=0$. We easily obtain
$\pm\psi(\hat\vt^\pm(\alpha))>0$ and
\[
\begin{split}
\frac{\d\psi}{\d\vartheta}
&= \alpha+\frac{U'(\vt)}{2U(\vt)}\left[\cotan(\vp_+-\vt)+\cotan(\vp_--\vt)\right]-2\\
&=\frac{U'(\vt)}{2U(\vt)}\,\frac{\sin(\psi+\pi)}{\sin(\vp_+-\vt)+\sin(\vp_--\vt)}-
(2-\alpha).
\end{split}
\]
We deduce that $\psi(\vt_0)=0$ implies $\d\psi(\vt_0)/\d\vt<0$, so that $\psi$ has
exactly one zero as claimed.

Let us come to the case in which $\vjump(x_\alpha)=0$. Using again Lemmas \ref{lem:unst_man_control}, \ref{lem:st_man_control} we have that both the
unstable trajectory and the stable one meet the line $\vp=\vt+\pi/2$ in exactly
one point. We deduce that, for some $\tau_* \leq \tau_{**}$
\[
\vt(\tau_*)=\hat\vt^-(\alpha),\qquad \vt(\tau_{**})=\hat\vt^+(\alpha).
\]
This, if also $\pjump(x_\alpha)=0$, immediately yields $\hat\vt^-(\alpha)
=\hat\vt^+(\alpha)$. On the other hand, let us assume that $\tau_{*}<\tau_{**}$.
Then, by minimality, the corresponding segment on the line $\vp=\vt+\pi/2$
must be traveled with $\vt$ monotone; since $\vt$ is $\cont^1$ and
$\vt'(\tau^*-)>0$, we deduce that $\vt'>0$ on $[\tau_{*},\tau_{**}]$,
i.e. $\hat\vt^-(\alpha)<\hat\vt^+(\alpha)$. Again, in both cases, the uniqueness of
$x_\alpha$ inside its category is due to the initial conditions $r(\tau_{*})=
r(\tau_{**})=1$.

Now the proof easily follows, indeed, in each of the two triplet of conditions,
at least one instance must occur and each one excludes the others.
\end{proof}
We are ready to prove our main theorem in the present case.
\begin{proof}[Proof of Theorem \ref{theo:main}, case $\pi < \vt^+ - \vt^- \leq 2\pi$]
As already mentioned, the first part of the theorem is a consequence of
Proposition \ref{propo:unique_baralfa} and Definition \ref{defi:baralfa}, while the
second easily follows by comparing Proposition \ref{propo:unique_baralfa} and the third
instance of Proposition \ref{propo:classification_constrained}. To prove
the last part we can use Lemma \ref{lem:exist_morse} in combination with Proposition
\ref{propo:classification_constrained}. In this way, we are left to show the existence
of two values $\alpha_1$, $\alpha_2$ such that the order between $\hat\vt^-(\alpha_i)$
and $\hat\vt^+(\alpha_i)$ is reversed by switching between $i=1$ and $i=2$. To this
aim, reasoning exactly as in the proof of Lemma \ref{lem:stimaalpha}, one can prove
the analogous of estimates \eqref{eq:stimadx}, \eqref{eq:stimasx}, that is
\[
\frac{2}{2-\alpha}\arcsin \sqrt\frac{U_{\min}}{U_{\max}}\leq\hat\vt^-(\alpha)-\vt^-\leq\frac{\pi}{2-\alpha},
\]
\[
\frac{2}{2-\alpha}\arcsin \sqrt\frac{U_{\min}}{U_{\max}}\leq\vt^+-\hat\vt^+(\alpha)\leq\frac{\pi}{2-\alpha}.
\]
Summing up and rearranging we obtain
\[
(\vt^+-\vt^-) - \frac{2\pi}{2-\alpha}\leq \hat\vt^+(\alpha)-\hat\vt^-(\alpha) \leq
(\vt^+-\vt^-) - \frac{4}{2-\alpha}\arcsin \sqrt\frac{U_{\min}}{U_{\max}}.
\]
It is now trivial, taking into account the limitations for $\vt^+-\vt^-$,
to verify that if $\alpha_1$ is small then
$\hat\vt^-(\alpha_1)< \hat\vt^+(\alpha_1)$, while if $\alpha_2$ is near $2$
then the opposite inequality holds.
\end{proof}
We conclude this section with a few words about the case $0<\vt^+-\vt^-\leq\pi$.
\begin{remark}
If $0<\vt^+-\vt^-\leq\pi$ then explicit conditions can be provided to show
that the number $\bar\alpha$, and hence parabolic minimizers, may or may not exist,
depending on the properties of $U$. For instance, if $U$ is a small perturbation
of a constant (that is, $V$ is an anisotropic small perturbation of an isotropic
potential), then $\bar\alpha$ does not exist, recall Figure \ref{fig:keplero}. On the
other hand, it is possible to construct angular potentials $U$ with arbitrarily
small $\vt^+-\vt^-$, such that the corresponding $\bar\alpha$ exists: roughly
speaking, this can be done by choosing $U$ very larger than $U_{\min}$ on a compact
subinterval of $(\vt^-,\vt^+)$, see Lemma 6.11 in \cite{BTV}.
\end{remark}

%===============================
\section{General Winding Number}\label{sec:srotolo}
%===============================

In the previous section we ruled out the case in which $\vt^+-\vt^- \in (\pi,2\pi]$.
This section is devoted to reformulate the case
\[
2h\pi < \vt^+ - \vt^- \leq 2(h+1)\pi, \quad h \geq 1
\]
in terms of that previous case, completing the proof of Theorem \ref{theo:main}
(again, the case $-2(h+1)\pi \leq \vt^+ - \vt^- < -2h\pi$ is easily treated
using time reversibility). This can be done using the following conformal change of variables.

\begin{lemma}\label{lem:conf}
Let $x=(r,\vt)$ be defined for $t \in [a,b]$, with $\min_t r >0$ and
$y=(\rho,\vp)$ be defined for $\tau \in [a',b']$, with $\min_{\tau} \rho >0$.
Let us assume that, for some $\beta>0$ there holds
\[
\tau = a' + \int_a^t r^{2(1-\beta)/\beta} \dt, \qquad
r(t) = \rho^\beta(\tau), \qquad
\vt(t) = \beta\vp(\tau),
\]
$b'=a'+\int_a^b r^{2(1-\beta)/\beta} \dt$.
Finally, let $U$ be $2\pi$-periodic and
\[
V(x) = \frac{U(\vt)}{r^\alpha}.
\]
Then
\[
\int_{a}^{b} \frac12 |\dot x|^2 + V(x)\dt = \beta^2 \int_{a'}^{b'} \frac12 |y'|^2 + \tilde V(y) \dtau,
\]
where
\[
\tilde V(y) = \frac{\tilde U(\vp)}{\rho^{\tilde\alpha}} \text{ with }
\tilde U(\vp) = \frac{U(\beta\vp)}{\beta^2}\text{ and }
\tilde \alpha = 2-\beta(2-\alpha).
\]
\end{lemma}

\begin{proof}
By direct computation we have
\[
V(x)=\frac{U(\vt)}{r^\alpha}=\frac{U(\beta\vp)}{\rho^{\alpha\beta}}
\]
and
\[
\begin{split}
|\dot x|^2 & = \dot r^2+r^2\dot\vt^2
= \beta^2\rho^{2\beta-2}\dot \rho^2 + \beta^2\rho^{2\beta}\dot \vp^2 \\
& = \beta^2\rho^{2\beta-2} \left[(\rho')^2+\rho^2(\vt')^2\right]\left(\frac{\dtau}{\dt}\right)^2
= \beta^2\rho^{2(1-\beta)}|y'|^2.
\end{split}
\]
Substituting in the action we obtain
\[
\int_{a}^{b} \frac12 |\dot x|^2 + V(x)\dt =
 \beta^2\int_{a'}^{b'}
 \left(\rho^{2(1-\beta)} \frac{|y'|^2}{2} + \frac{U(\beta\vp)/\beta^2}{\rho^{\alpha\beta}}\right)
 \cdot \rho^{-2(1-\beta)} \dtau. \qedhere
\]
\end{proof}
\begin{remark}\label{rem:regU}
It is immediate to show that if $U \in \Uh_{\vt_1 \vt_2}$, $\vt^{\pm} \in \Theta_{\vt_1 \vt_2}$,
and $\tilde U$ is defined as in the previous lemma,
then $\tilde U \in \Uh_{\frac{\vt_1}{\beta} \frac{\vt_2}{\beta}}$
and $\frac{\vt^{\pm}}{\beta} \in \Theta_{\frac{\vt_1}{\beta} \frac{\vt_2}{\beta}}$.
\end{remark}
We are in a position to conclude the proof of Theorem \ref{theo:main}. This is
done through the following proposition.
\begin{proposition}
Let $2h\pi < \vt^+ - \vt^- \leq 2(h+1)\pi$ for some $h\geq 1$ and let us define
\[
\tilde{\vt}^{\pm} = \frac{\vt^{\pm}}{h+1} \text{ and }
\tilde U(\vt) = \frac{U((h+1)\vt)}{(h+1)^2}.
\]
Then $\pi < \tilde\vt^+ - \tilde\vt^- \leq 2\pi$ and
\[
\bar \alpha (\vt^-,\vt^+,U) =
2-\frac{2-\bar \alpha(\tilde\vt^-,\tilde\vt^+,\tilde U)}{h+1},
\]
the latter being well defined by Section \ref{sec:sector}.
\end{proposition}
\begin{proof}
We have to show that $(U,\alpha)$, $\alpha \in (0,2)$, admits a parabolic Morse minimizer if and only if
$\alpha$ is equal to the r.h.s. of the expression above.
To start with we observe that, if
\[
\alpha \leq 2-\frac1h
\]
then $(U,\alpha)$ can not admit a parabolic Morse minimizer.
Indeed, on the contrary, Lemma \ref{lem:stimaalpha} would apply, yielding
\[
2-\frac{2\pi}{\vt^+-\vt^-} \leq \alpha,
\]
in contradiction with the fact that $\vt^+-\vt^- > 2h\pi$.
On the other hand, if $\alpha > 2-1/h$, we can apply
Lemma \ref{lem:conf} and Remark \ref{rem:regU},
obtaining that trajectories connecting $\vt^{\pm}$ with potential $(U,\alpha)$
correspond to trajectories connecting $\tilde\vt^{\pm}$ with potential $(\tilde U,\tilde \alpha)$,
with $\tilde \alpha = 2-(h+1)(2-\alpha)$.
As a consequence, in order to prove the proposition, we simply have to show that
the results of Section \ref{sec:sector} can be applied to this latter context.
To this aim, the only non-immediate thing to check is that $\tilde \alpha \in (0,2)$.
This is easily proved by monotonicity, since
\[
2-\frac1h <\alpha < 2 \quad \implies \quad 1-\frac{1}{h} < \tilde \alpha < 2. \qedhere
\]
\end{proof}
%=======================================
\section{Proof of Theorems \ref{theo:bolza} and \ref{theo:gordon}}\label{sec:other_proofs}
%=======================================
The strategy in the proof of both theorems is the following. To start with we assume by contradiction
the existence of a colliding minimizer and we study a class of constrained minimization problems,
restricting to the paths having distance from the origin at least $\eps$.
Next we let $\eps\to0$ and perform a blow-up procedure obtaining as a limit a global zero-energy path,
which connects two central configurations at $r \to \infty$ and solves the equation outside the constraint.
Finally, we obtain a contradiction to the existence of such a path by exploiting the results obtained in the previous sections.
To this last aim a crucial tool is given by the following lemma,
which is a generalization of Proposition
\ref{propo:no_para_in_strip} to fixed-time constrained minimizers
with $\vjump=0$ connecting (not necessarily minimal) central configurations.
\begin{lemma}\label{lemma:barriere}
Let $(U,\alpha)$ be fixed and $x=(r,\vt) \in H^1_{\mathrm{loc}}(\RR)$
be such that, for some $t_*\leq 0 \leq t_{**}$, it holds
\begin{itemize}
\item $x$ is $\cont^1$ and it is minimal under fixed-time variations;
\item $|x| \to \infty$ and $x/|x| \to \widetilde \vt^{\pm}$ as $t \to \pm \infty$;
\item $r(t)\equiv 1$ if and only if $t \in [t_*,t_{**}]$, $\dot r(t)<0$ (resp. $\dot r(t)>0$)
      if and only if  $t<t_*$ (resp. $t>t_{**}$);
\item $x$ solves \eqref{eq:dynsys} for $t \notin [t_*,t_{**}]$ and \eqref{eq:dynsys2} for every $t$;
\item there exist $\vt^{\pm}$ minimal central configurations such that
      $[\widetilde \vt^-,\widetilde \vt^+] \subset [\vt^-,\vt^+]$.
\end{itemize}
Then $\alpha \leq \bar\alpha(\vt^-,\vt^+,U)$.
\end{lemma}
\begin{proof}
Reasoning as in Remark \ref{rem:anotherfuckingkeyremark},
we can project $x$ to the Devaney's phase plane. As usual, the corresponding graph consists
in the junction of three arcs in the strip:
the part of an unstable trajectory emanating from (say) $(\widetilde\vt^-, \widetilde\vt^-+\pi)$
up to $A$, its crossing point with the straight line $\varphi=\vt+\pi/2$;
the arc of the stable manifold entering
in $(\widetilde\vt^+, \widetilde\vt^+)$ back to $B$, its crossing point with the same straight line;
a segment joining the two crossings, which is traveled monotonically in $\vt$ by minimality.
Since $\vt$ must be $\cont^1$ across the whole junction, and trajectories of \eqref{eq:dydthetaphi}
cross the line $\varphi=\vt+\pi/2$ with increasing $\vt$, we infer that
\[
\vt_A \leq \vt_B.
\]
On the other hand since the whole junction is completely contained in the strip $[\vt^-,\vt^+]$,
Corollary \ref{coro:barriere} implies that
\[
\vt_A \geq \hat\vt^-(\alpha), \qquad
\vt_B \leq \hat\vt^+(\alpha),
\]
and the conclusion follows from the definition of $\bar \alpha$.
\end{proof}
\begin{remark}\label{rem:baralpha}
In the previous lemma $\alpha =\bar \alpha$ forces $\vt_A=\vt_B$ and hence
$\widetilde\vt^{\pm} = \vt^{\pm}$.
\end{remark}
\begin{proof}[Proof of Theorem \ref{theo:bolza}]
Taking advantage of the conformal
equivariance of the problem, arguing as in Section \ref{sec:sector} we can
reduce to the case $\vt^+ \leq \vt^-+2\pi$. We argue by contradiction, assuming that
for some $x_1=(r_1,\vp_1)$, $x_2=(r_2,\vp_2)$ in the sector $[\vt^-,\vt^+]$ and $t_1<t_2$ there exists a
Bolza minimizer completely contained in the sector and traveling through the origin.
As we did in Definition \ref{defi:constr_Morse_min}
for Morse minimizers, we can introduce the notion of constrained Bolza ones. More precisely,
let us consider the set of paths within the sector having the required endpoints:
\[
\Gamma:=\left\{x=(r,\vt):\,r(t_i)=r_i,\,\vt(t_i)=\vp_i,\,\vt(t)\in[\vt^-,\vt^+]
\text{ for }t\in[t_1,t_2]\right\};
\]
next we consider a small parameter $\eps>0$ and we compare the values of the two
following constrained minimization problems: the one featuring equality constraint
\begin{equation*}
c_\eps^c:=\min\{\action(t_1,t_2;x):\,x\in\Gamma\text{ and }\min_{[t_1,t_2]}r(t)=\eps\}
\end{equation*}
with the obstacle-type problem
\begin{equation*}
c_\eps^c:=\min\{\action(t_1,t_2;x):\,x\in\Gamma\text{ and }\min_{[t_1,t_2]}r(t)\geq\eps\}
\end{equation*}
(it is standard to prove that they are both achieved).
Of course, $c_\eps$ is non decreasing in $\eps$ and $c_\eps\leq c_\eps^c$, $\forall\eps>0$.
Arguing as in the proof of Theorem 18 in \cite{TerVen07}, if $c_\eps< c_\eps^c$
for every small positive $\eps$, then we are done. Hence, we can reduce our analysis to
the case of a vanishing sequence $\eps_n\to 0$ with $c_{\eps_n}=c_{\eps_n}^c$
and such that the two constrained minimization problems share the same class of minimizers.
Let us take a sequence $x_n$ of such minimizers: they can interact with the constraints in essentially two ways.
On one hand, they are $\cont^1$ when they touch the lines $\vt = \vp_i$; one the other hand, concerning the
circular constraint as in Section \ref{sec:sector}
we may have either $\vjump(x_n)>0$ or $\vjump(x_n)=0$
(it can be shown that the classification in terms of position and velocity jumps holds also
for fixed-time minimizers, at least for $\eps$ small, see also \cite{BTV}, Proposition 3.6).
It is immediate to rule out the case $\vjump(x_n)>0$, because a local variation can be easily
produced in contradiction with the fact that $c_{\eps_n}=c_{\eps_n}^c$.
Following again the argument of the proof of Proposition 20 of \cite{TerVen07},
one sees that the energies are uniformly bounded along the sequence.
Defining the blow-up sequence
\[
\hat x_n(t) = \frac{1}{\eps_n}x_n(\eps_n^{-\frac{2+\alpha}{2}}t)
\]
we can argue as in \cite{TerVen07} (pages 486--488) to pass to the limit and find a zero-energy $\cont^1$-path,
minimal under fixed-time variations for the homogeneous potential $(U,\alpha)$.
We observe that such paths can not touch the lines $\vt=\vp_i$: indeed, it would be a $\cont^1$ junction,
in contradiction to the uniqueness for Cauchy problems.
As a consequence the blow-up limit consists of a pair of parabolic arcs,
connected by a circular arc, within the sector $(\vt^-,\vt^+)$.
The two parabolic arcs have ingoing and outgoing asymptotic central configurations
$\widetilde\vt^-$, $\widetilde\vt^+$ such that $\vt^- \leq \widetilde\vt^-<\widetilde\vt^+ \leq \vt^+$.
Since $\alpha >\bar \alpha$ this contradicts Lemma \ref{lemma:barriere}.
\end{proof}
\begin{remark}\label{rem:baralpha2}
The previous proof, together with Remark \ref{rem:baralpha}, immediately
provides Proposition \ref{propo:1.5}. Moreover, it is possible to show that,
if $\alpha=\bar\alpha$, then any Bolza minimizer within the sector either
is collisionless or it collides with ingoing/outgoing directions precisely
$\vt^-$ and $\vt^+$.
\end{remark}
\begin{proof}[Proof of Theorem \ref{theo:gordon}]
First of all we can take advantage of the conformal invariance to reduce to the case $k=1$.
Next we set again the constrained minimization problems over the set of loops winding one time
around the origin:
\begin{equation*}
c_\eps^c(\alpha,U)=\min\{\action(0,T;x)\;;\;x(0)=x(T)\;,\;\text{deg}(x,0)=1\;\text{and}\; \min_{[0,T]}r(t)=\eps\}
\end{equation*}
(here $\text{deg}(x,0)$ denotes the topological degree of the map $x$).  We also set
\[
c^c=\liminf_{\eps\to 0}c_\eps^c.
\]
It is easy to prove that a minimizing periodic trajectory in this class corresponds to a simple loop.
We remark that, under the previous notation, our aim is to prove that there exists $\eps>0$ such that
$c_\eps^c < c^c$.
This will be done in two steps.\\
{\em Step 1: If every maximum of $U$ satisfies condition \eqref{eq:too_strict}
then there exists $\eps>0$ such that $c_\eps^c(\alpha,U)\leq c^c$}.
Indeed, if not, we would have $c_\eps^c> c^c$ for all positive $\eps$ and
hence, for every small $\eps_2>0$, we can find a smaller $\eps_1$ such that
\[
c_{\eps_1,\eps_2}=\min\{\action(0,T;x)\;;\;x(0)=x(T)\;,\;\text{deg}(x,0)=1\;\text{and}\; \min_{[0,T]}r(t)\in[\eps_1,\eps_2)\}
\]
is achieved. In this way, we find the existence of a fixed--time constrained minimizing trajectory with $\vjump=0$.
Reasoning again as in \cite{TerVen07}, letting $\eps_2\to 0$ and going to a blow--up sequence,
we find in the limit a parabolic fixed--time constrained minimizing trajectory with $\vjump=0$.
Now we look at its asymptotic central configurations and we go to the phase plane.
We have to deal with the case when the corresponding trajectory connects a pair of
stationary points $(\widetilde\vt^-, \widetilde\vt^-+\pi)$ and $(\widetilde\vt^+,\widetilde\vt^+)$ and,
by the absence of self intersections,  we infer $\widetilde\vt^+\leq \widetilde\vt^-+2\pi$.
Now, if $\widetilde \vt^-$ is a maximum for $U$, then
thanks to Corollary \ref{coro:maximal_not_minimal}, we reach a contradiction.
On the other hand, if $\widetilde \vt^-$ is a minimum, we can apply Lemma \ref{lemma:barriere}
with $[\vt^-,\vt^+]:=[\widetilde\vt^-,\widetilde\vt^-+2\pi]$ and obtain a contradiction
with the fact that $\alpha>\bar \alpha(\widetilde\vt^-,\widetilde\vt^-+2\pi,U)$.

\noindent
{\em Step 2: if $U$ and $\widetilde U$ share the same global minimizers, at the same level $U_{min}$,
then $c^c(\alpha,U)=c^c(\alpha,\widetilde U)$}.
Indeed let $(r(t),\vt(t))$ achieve $c^c(\alpha,U)$, then also $(r(t),\vt^*)$, for any $\vt^*$ minimal configuration for $U$,
achieves the same level. On this last path the actions with potentials $U$ and $\widetilde U$ coincide.
Therefore $c^c(\alpha,\widetilde U) \leq c^c(\alpha,U)$; the claim follows by exchanging the roles of $U$ and $\widetilde U$.

\noindent
{\em Step 3: conclusion}.
Let $U$ satisfy the assumptions of the theorem.
We can always construct another Morse potential $\widetilde U\in\Uh$, still satisfying \eqref{eq:alfamax},
such that $\min \widetilde{U}= \min U$, $\widetilde U \geq U$, $\widetilde U\neq U$ and, last but not least,
$\widetilde U$ satisfies \eqref{eq:too_strict}.
Now, by Step 1, there exists $\eps>0$ such that $c^c_\eps(\alpha,\widetilde{U}) \leq c^c(\alpha,\widetilde{U})$,
the former being achieved by a collisionless loop $\widetilde{x}$.
Evaluating the action relative to $U$ along $\widetilde{x}$ we obtain
\[
c^c_\eps(\alpha,U) < c^c_\eps(\alpha,\widetilde{U}) \leq  c^c(\alpha,U),
\]
as was to be shown.
\end{proof}

%
%=======================================

%=======================================

\begin{thebibliography}{10}

\bibitem{AC-Z}
A.~Ambrosetti and V.~Coti~Zelati.
\newblock {\em Periodic solutions of singular {L}agrangian systems}.
\newblock Progress in Nonlinear Differential Equations and their Applications,
  10. Birkh\"auser Boston Inc., Boston, MA, 1993.

\bibitem{AriGazTer2000}
G.~Arioli, F.~Gazzola, and S.~Terracini.
\newblock Minimization properties of {H}ill's orbits and applications to some
  {$N$}-body problems.
\newblock {\em Ann. Inst. H. Poincar\'e Anal. Non Lin\'eaire}, 17(5):617--650,
  2000.

\bibitem{BFT2}
V.~Barutello, D.~L. Ferrario, and S.~Terracini.
\newblock On the singularities of generalized solutions to {$n$}-body-type
  problems.
\newblock {\em Int. Math. Res. Not. IMRN}, pages Art. ID rnn 069, 78pp, 2008.

\bibitem{BaSe}
V.~Barutello and S.~Secchi.
\newblock Morse index properties of colliding solutions to the {$N$}-body
  problem.
\newblock {\em Ann. Inst. H. Poincar\'e Anal. Non Lin\'eaire}, 25(3):539--565,
  2008.

\bibitem{BTV}
V.~Barutello, S.~Terracini, and G.~Verzini.
\newblock Entire parabolic trajectories as minimal phase transitions.
\newblock Preprint, arXiv:1105.3358v1 [math.DS], 2011.

\bibitem{Chen2001}
K-C. Chen.
\newblock Action-minimizing orbits in the parallelogram four-body problem with
  equal masses.
\newblock {\em Arch. Ration. Mech. Anal.}, 158(4):293--318, 2001.

\bibitem{Chen2008}
K-C. Chen.
\newblock Existence and minimizing properties of retrograde orbits to the
  three-body problem with various choices of masses.
\newblock {\em Ann. of Math. (2)}, 167(2):325--348, 2008.

\bibitem{chencinerICM02}
A.~Chenciner.
\newblock Action minimizing solutions of the {N}ewtonian {$n$}-body problem:
  from homology to symmetry.
\newblock In {\em Proceedings of the {I}nternational {C}ongress of
  {M}athematicians, {V}ol. {III} ({B}eijing, 2002)}, pages 279--294, Beijing,
  2002. Higher Ed. Press.

\bibitem{ChMont1999}
A.~Chenciner and Montgomery R.
\newblock A remarkable periodic solution of the three body problem in the case
  of equal masses.
\newblock {\em Ann. of Math.}, 152 3:881--901, 1999.

\bibitem{ChenVen2000}
A.~Chenciner and A.~Venturelli.
\newblock Minima de l'int\'egrale d'action du probl\`eme newtonien de 4 corps
  de masses \'egales dans {${\bf R}^3$}: orbites ``hip-hop''.
\newblock {\em Celestial Mech. Dynam. Astronom.}, 77(2):139--152 (2001), 2000.

\bibitem{CZelSer1992}
V.~Coti~Zelati and E.~Serra.
\newblock Some properties of collision and noncollision orbits for a class of
  singular dynamical systems.
\newblock {\em Atti Accad. Naz. Lincei Cl. Sci. Fis. Mat. Natur. Rend. Lincei
  (9) Mat. Appl.}, 3(3):217--222, 1992.

\bibitem{CZelSer1994}
V.~Coti~Zelati and E.~Serra.
\newblock Collision and non-collision solutions for a class of {K}eplerian-like
  dynamical systems.
\newblock {\em Ann. Mat. Pura Appl. (4)}, 166:343--362, 1994.

\bibitem{LuzMad2011}
A.~da~Luz and E.~Maderna.
\newblock On the free time minimizers of the newtonian n-body problem.
\newblock {\em Math. Proc. Cambridge Philos. Soc.}, to appear, 2011.

\bibitem{DegGiaMar1988}
M.~Degiovanni, F.~Giannoni, and A.~Marino.
\newblock Dynamical systems with {N}ewtonian type potentials.
\newblock {\em Atti Accad. Naz. Lincei Rend. Cl. Sci. Fis. Mat. Natur. (8)},
  81(3):271--277 (1988), 1987.

\bibitem{DegGiaMar1987}
M.~Degiovanni, F.~Giannoni, and A.~Marino.
\newblock Periodic solutions of dynamical systems with {N}ewtonian type
  potentials.
\newblock In {\em Periodic solutions of {H}amiltonian systems and related
  topics ({I}l {C}iocco, 1986)}, volume 209 of {\em NATO Adv. Sci. Inst. Ser. C
  Math. Phys. Sci.}, pages 111--115. Reidel, Dordrecht, 1987.

\bibitem{DevInvMath1978}
R.~L. Devaney.
\newblock Collision orbits in the anisotropic {K}epler problem.
\newblock {\em Invent. Math.}, 45(3):221--251, 1978.

\bibitem{DevProgMath1981}
R.~L. Devaney.
\newblock Singularities in classical mechanical systems.
\newblock In {\em Ergodic theory and dynamical systems, {I} ({C}ollege {P}ark,
  {M}d., 1979--80)}, volume~10 of {\em Progr. Math.}, pages 211--333.
  Birkh\"auser Boston, Mass., 1981.

\bibitem{Fer2007}
D.~L. Ferrario.
\newblock Transitive decomposition of symmetry groups for the {$n$}-body
  problem.
\newblock {\em Adv. Math.}, 213(2):763--784, 2007.

\bibitem{FT2003}
D.~L. Ferrario and S.~Terracini.
\newblock On the existence of collisionless equivariant minimizers for the
  classical {$n$}-body problem.
\newblock {\em Invent. Math.}, 155(2):305--362, 2004.

\bibitem{Gordon77}
W.~B. Gordon.
\newblock A minimizing property of {K}eplerian orbits.
\newblock {\em Amer. J. Math.}, 99(5):961--971, 1977.

\bibitem{Gutzwiller73}
M.~C. Gutzwiller.
\newblock The anisotropic {K}epler problem in two dimensions.
\newblock {\em J. Mathematical Phys.}, 14:139--152, 1973.

\bibitem{Gutzwiller77}
M.~C. Gutzwiller.
\newblock Bernoulli sequences and trajectories in the anisotropic {K}epler
  problem.
\newblock {\em J. Mathematical Phys.}, 18(4):806--823, 1977.

\bibitem{Gutzwiller81}
M.~C. Gutzwiller.
\newblock Periodic orbits in the anisotropic {K}epler problem.
\newblock In {\em Classical mechanics and dynamical systems ({M}edford,
  {M}ass., 1979)}, volume~70 of {\em Lecture Notes in Pure and Appl. Math.},
  pages 69--90. Dekker, New York, 1981.

\bibitem{MadVen2009}
E.~Maderna and A.~Venturelli.
\newblock Globally minimizing parabolic motions in the {N}ewtonian {$N$}-body
  problem.
\newblock {\em Arch. Ration. Mech. Anal.}, 194(1):283--313, 2009.

\bibitem{Marchal01}
C.~Marchal.
\newblock How the method of minimization of action avoids singularities.
\newblock {\em Celestial Mech. Dynam. Astronom.}, 83(1-4):325--353, 2002.
\newblock Modern celestial mechanics: from theory to applications (Rome, 2001).

\bibitem{Mont1998}
R.~Montgomery.
\newblock The {$N$}-body problem, the braid group, and action-minimizing
  periodic solutions.
\newblock {\em Nonlinearity}, 11(2):363--376, 1998.

\bibitem{SerraTer94}
E.~Serra and S.~Terracini.
\newblock Noncollision solutions to some singular minimization problems with
  {K}eplerian-like potentials.
\newblock {\em Nonlinear Anal.}, 22(1):45--62, 1994.

\bibitem{Tanak2000}
K.~Tanaka.
\newblock Periodic solutions for singular {H}amiltonian systems and closed
  geodesics on non-compact {R}iemannian manifolds.
\newblock {\em Ann. Inst. H. Poincar\'e Anal. Non Lin\'eaire}, 17(1):1--33,
  2000.

\bibitem{TerVen07}
S.~Terracini and A.~Venturelli.
\newblock Symmetric trajectories for the {$2N$}-body problem with equal masses.
\newblock {\em Arch. Ration. Mech. Anal.}, 184(3):465--493, 2007.

\bibitem{Whittaker}
E.~T. Whittaker.
\newblock {\em A treatise on the analytical dynamics of particles and rigid
  bodies: {W}ith an introduction to the problem of three bodies}.
\newblock 4th ed. Cambridge University Press, New York, 1959.

\end{thebibliography}
\end{document}